\documentclass[11pt]{amsart}
\usepackage {fancybox}
\usepackage{graphicx}
\usepackage{ascmac}
\usepackage{tikz-qtree}
\usepackage{forest}
\usepackage{mathrsfs}
\usepackage{amscd,verbatim}
\usepackage{floatflt}
\usepackage{amssymb}
\usepackage{tikz-cd}
\usepackage[all]{xy}
\usepackage[hdivide={20mm,,20mm}, vdivide={20mm,,20mm}]{geometry}
\usepackage{stmaryrd} 
\usepackage[colorlinks,linkcolor=blue,citecolor=blue,urlcolor=red]{hyperref}
\usepackage[normalem]{ulem}

\usepackage{mathtools}

\renewcommand{\AA}{\mathbb{A}}

\newcommand{\N}{\mathbb{N}}
\newcommand{\PP}{\mathbb{P}}
\newcommand{\Q}{\mathbb{Q}}
\newcommand{\Z}{\mathbb{Z}}

\newcommand{\sF}{\mathcal{F}}
\newcommand{\sG}{\mathcal{G}}

\newcommand{\sO}{\mathcal{O}}

\newcommand{\sQ}{\mathcal{Q}}

\newcommand{\Cone}{\operatorname{Cone}}

\newcommand{\Chow}{\operatorname{\mathbf{Chow}}}
\newcommand{\KMM}{\operatorname{\mathbf{KMM}}}
\newcommand{\KM}{\operatorname{\mathbf{KM}}}

\newcommand{\DM}{\operatorname{\mathbf{DM}}}

\newcommand{\Spec}{\operatorname{Spec}}
\newcommand{\Proj}{\operatorname{\mathbf{Proj}}}
\newcommand{\Sm}{\operatorname{\mathbf{Sm}}}

\newcommand{\eff}{{\operatorname{eff}}}

\newcommand{\ch}{{\operatorname{ch}}}

\newcommand{\CH}{{\operatorname{CH}}}

\renewcommand{\lim}{\operatornamewithlimits{\varprojlim}}

\newcommand{\ol}{\overline}

\renewcommand{\epsilon}{\varepsilon}

\newcounter{spec}
{\end{list}}%

\newtheorem{lemma}{Lemma}[section]

\newtheorem{thm}[lemma]{Theorem}

\newtheorem{prop}[lemma]{Proposition}

\newtheorem{cor}[lemma]{Corollary}

\theoremstyle{definition}

\theoremstyle{remark}

\newtheorem{example}[lemma]{Example}

\numberwithin{equation}{section}

\title[Derived invariants and motives, Part \rm\,I\,]{Derived invariants and motives, Part \rm\,I\, Integral Grothendieck Riemann-Roch and non-commutative motives}\author{Keiho Matsumoto}

\begin{document}
\maketitle

\begin{abstract}
The goal of this series of papers is to give a new non-commutative approach to problems about the density of reductions such as the conjecture of Joshi-Rajan, and the generalization of the conjecture of Serre. In this paper, we prove integral Grothendieck Riemann-Roch which was proved by Pappas in the case $\ch(k)=0$. As a corollary we prove an integral analogue of Kontsevich's comparison theorem, and we show that if a smooth projective variety $X$ has a full exceptional collection then there is an explicit formula of the motive of $X$ up to bounded torsion.
\end{abstract}

Keywords: motives, derived category, non-commutative algebraic geometry

2020 Mathematics Subject Classification: Primary 14C15, 14F08

\section{Introduction}
The goal of this series of papers is to give a new non-commutative approach to problems about the density of reductions such as the conjecture of Joshi-Rajan \cite[Conj.4.1.1]{Joshi}, and a generalization of the conjecture of Serre \cite[Conj.3.1.1]{Joshi}. The density of reductions is related to the representation theorem of algebraic groups via the global Langlands problem. In this article, we are interested in the comparison between Chow motives and non-commutative motives. With $\Q$-coefficient, the comparison was settled by Kontsevich \cite{Kont09} and Tabuada \cite{Tab14}. They showed that there is a fully faithful functor from the $\Q(1)$-orbit category of Chow motives with $\Q$-coefficients to the category of non-commutative motives with $\Q$-coefficients (see \cite{Tab14}).
This fully faithful functor arises from the Grothendieck Riemann-Roch theorem \cite{BorelSerre}, \cite{ManinGRR}. The purpose of this paper is to prove an integral analogue of it via integral Grothendieck Riemann-Roch theorem. 

To formulate the Grothendieck Riemann-Roch theorem \cite{BorelSerre}, Grothendieck uses the Chern character $\ch(E)$ and Todd class $Td(E)$ of a vector bundle $E$ on a smooth variety. For a smooth projective variety $X$ of dimension $d$ over a field $k$ and a vector bundle $E$ on $X$, we will see that the Chern character $l!\cdot \ch(E)$ and Todd class $T_{l}\cdot Td(E)$ are well-defined in the Chow group $\CH^*(X)$ of $X$ with integral-coefficients for any $l \geq d$ where we set $T_m:=\displaystyle\prod_{p:\text{prime number}} p^{[\frac{m}{p-1}]}$.

For natural numbers $d$ and $e\geq d$, we denote the category of smooth projective varieties of dimension less than or equal to $d$ that can be embedded in the projective space $\PP^e_k$ by $\Sm\Proj^{\leq d}_{(e)}(k)$. 

Our main goal is the following theorem.

\begin{thm}\label{GRRfinite}(Integral Grothendieck Riemann-Roch theorem)
Let $f:X \to Y$ be a projective morphism of smooth projective varieties over a field $k$. We assume $X,Y \in \Sm\Proj^{\leq d}_{(e)}(k)$. Then for all $x\in K_0(X)$ there exist an equation
\[
 f_*\bigl( l!^2\cdot (T_l)^2 \cdot \ch(x)\cdot Td(T_X)\bigr) = l!^2\cdot (T_l)^2 \cdot \ch(f_*x)\cdot Td(T_Y)
\]
in $\CH^*(Y)$ for any $l\geq d+e$.
\end{thm}
In the case $\ch(k)=0$, Theorem~\ref{GRRfinite} is can be deduced from Pappas's integral Grothendieck Riemann-Roch \cite[Theorem~2.2]{Papas07}. His proof use resolution of singularities and Bertini's theorem thus is not available in positive characteristic. We recall his theorem as Theorem~\ref{PapasintegralGRR} in the appenIn thdix.

Let us mention applications of our integral Grothendieck Riemann-Roch theorem. We denote by $\KM(k)$ the category of Gillet-Soul{\'e}'s $K$-motives (see \cite{KMOTIVE}, or our Section~\ref{section:IntKonThm}), and we denote by $\KM^{\leq d}_{(e)}(k)$ the smallest full subcategory of $\KM(k)$ which contains the image of the functor $\Sm\Proj^{\leq d}_{(e)}(k) \to \KM(k)$ and is closed under finite coproducts. Recall $\KM(k)$ is the idempotent completion of a category whose objects are smooth projective varieties, and for smooth projective $X$ and $Y$, we have $\hom_{\KM(k)}(X, Y) = K_0(X\times Y)$ (see section~\ref{section:IntKonThm}). We denote by $\KMM(k)$ the category of non-commutative motives (see \cite{Tab14}). 

\begin{cor}(Integral analogue of Kontsevich`s comparison theorem)
We assume that the base field $k$ is perfect. For natural numbers $d$ and $e$, $l\geq 2d+e$, there is a ${\Z[\frac{1}{(l+1)!}]}$-linear functor 
\[
\Phi_{\Z[\frac{1}{(l+1)!}]}:\KM^{\leq d}_{(e)}(k)_{\Z[\frac{1}{(l+1)!}]} \to \Chow(k)_{\Z[\frac{1}{(l+1)!}]}/-\otimes T_{\Z[\frac{1}{(l+1)!}]}
\]
where $T_{\Z[\frac{1}{(l+1)!}]}$ is the Tate motive. The functor $\Phi_{\Z[\frac{1}{(l+1)!}]}$ making the following diagram commute:
\[
\xymatrix{
{\textbf{dgcat}}(k) \ar[d]_{U}&&\Sm\Proj^{\leq d}_{(e)}(k)\ar[d]  \ar[ll]_{perf_{dg}(-)} & & \\
\KMM(k)_{\Z[\frac{1}{(l+1)!}]} &&\KM^{\leq d}_{(e)}(k)_{\Z[\frac{1}{(l+1)!}]} \ar[rr]_-{\Phi_{{\Z[\frac{1}{(l+1)!}]}}} \ar[ll]^{\theta} & &  \Chow(k)_{\Z[\frac{1}{(l+1)!}]}/-\otimes T_{\Z[\frac{1}{(l+1)!}]}
}
\]
where $\theta$ is the natural fully faithful functor. After applying $-\otimes \Q$, the lower row is compatible with the fully faithful functor $\Chow(k)_{\Q}/-\otimes T_{\Q} \to \KMM(k)_{\Q}$ described in \cite[section~8]{Tab14}.
\end{cor}

This corollary induces a generalization of Orlov's work on motives and derived categories \cite{Orlov04}.

\begin{cor}
For smooth projective varieties $X,Y$ in $\Sm\Proj^{\leq d}_{(e)}(k)$and $l\geq 2d+e$, if there is a fully faithful triangulated $k$-linear functor $D^b(X)\hookrightarrow D^b(Y)$, then there is a split injective morphism
\[
M(X)_{\Z[\frac{1}{ (l+1)!}]}(d_Y)[2d_Y] \hookrightarrow \bigoplus_{i=0}^{d_X+d_Y}M(Y)_{\Z[\frac{1}{(l+1)!}]}(i)[2i]
\]
in $\DM(k,\Z[\frac{1}{(l+1)!}])$.
\end{cor}

By using the decomposition of a non-commutative motive of smooth proper variety which has a full exceptional collection (see \cite[Lemma~5.1]{MarcolliTabuada}), we obtain the following.

\begin{cor}
For a smooth projective variety $X$ in $\Sm\Proj^{\leq d}_{(e)}(k)$ and $l\geq 2d+e$, if the twisted derived category of $(X)$ has a full exceptional collection $<E_1,E_2,..,E_m>$, then there  exist  integers $r_1,r_2,...,r_m\in \{0,...,\dim X\}$ giving rise to a canonical isomorphism in $\DM^\eff(k)$:
\[
M(X)_{\Z[\frac{1}{ (l+1)!}]}\simeq \bigoplus_{i=1}^m{\Z[\frac{1}{ (l+1)!}]}(r_i)[2r_i].
\]
\end{cor}
We summarise the strategy of the proof of Theorem~\ref{GRRfinite}. In section~\ref{21} and section~\ref{22}, we recall some of the properties of the Grothendieck group and the definition of Chern character and Todd class. In section~\ref{23}, we define integral Chern character and integral Todd class, and in section~\ref{24} we prove some properties of these such as the additive formula \eqref{nakanosine}, and the tensor product formula~\eqref{Dentsuusine}. In section~\ref{26}, we prove the integral Grothendieck Riemann-Roch for sections of projective bundles. In section~\ref{GRRproj}, we prove integral Grothendieck Riemann-Roch for projections from trivial projective bundles. In section~\ref{28}, we recall deformation to the normal cone, and use it to prove the integral Grothendieck Riemann-Roch for closed embeddings. Finally, in \ref{29}, we combine the above special cases to prove the integral Grothendieck Riemann-Roch.

\subsection{Notation and conventions}
We consider the following categories, and ring associated to  a field $k$, natural numbers $d,e,r$, and a commutative ring $R$. Where ever possible, we have used notation already existing in the literature.
\begin{itemize}
    \item[] $\Sm\Proj^{\leq d}_{(e)}(k)$: the full subcategory of $\Sm\Proj(S)$ whose objects can be embedded in $\PP^e_k$ and dimension of objects is less than or equal to $d$
   { \item[]{\[R_{r}= \Z[\frac{1}{r!}]\]}}
   { \item[]{$\KM^{\leq d}_{(e)}(k)$: the smallest full subcategory of $\KM(k)$ (see section~\ref{section:IntKonThm}) which contains the image of the functor $\Sm\Proj^{\leq d}_{(e)}(k) \to \KM(k)$ and is closed under finite coproducts}}
   {\item[]{$\Chow/-\otimes T$: is the orbit category (see Section~\ref{section:IntKonThm}).}} 
   { \item[]{$\Chow(k) \overset{\pi}{\to} \Chow(k)/-\otimes T$: the natural functor from $\Chow(k)$ to $\Chow(k)/-\otimes T$.}}
\end{itemize}

\section{Integral Grothendieck Riemann-Roch theorem} We begin by recalling some of the properties of the Grothendieck group of vector bundles and coherent sheaves. This material is contained in Fulton's book \cite{Fulton}.

\subsection{Grothendieck group of vector bundles and coherent sheaves}\label{21}
 For any scheme $X$, we write $K^0(X)$ for the Grothendieck group of locally free sheaves, $K_0(X)$ denotes the Grothendieck group of coherent sheaves. There is a canonical morphism
\[
K^0(X) \to K_0(X)
\]
which takes a vector bundle to its sheaf of sections. When $X$ is smooth variety over a field $k$ then this map is an isomorphism. For any map $f:X \to Y$ of schemes, the pullback of bundles induces a homomorphism
\[
f^*:K^0(Y) \to K^0(X).
\]
For a proper morphism $f:X \to Y$, there is a homomorphism
\[
f_*:K_0(X) \to K_0(Y)
\]
which take $[\sF]$ to $\Sigma_{i\geq 0}(-1)^i[R^if_*\sF]$. The tensor product of bundles makes $K^0(X)$ a ring and makes $K_0(X)$ a $K^0(X)$-module:
\[
K^0(X) \otimes K_0(X) \to K_0(X)
\]
where $[E]\cdot [\sF]$ goes to $[E\otimes_{\sO_X}\sF]$. As with Chow groups, the projection formula holds:
\[
f_*(f^*y \cdot x) =y \cdot f_*x \in K_0(Y)
\]
for a proper map $f:X\to Y$ and $x\in K_0(X)$, $y\in K^0(Y)$.

\subsection{Chern class}\label{22}

Now recall some fundamental facts about Chern classes of vector bundles, and the definition of the Chern character and Todd class. 

Let $E$ be a vector bundle on a scheme $Y$, suppose $f:X \to Y$ is a proper morphism. Then the projection formula 
\begin{equation}\label{projection}
f_*(c_i(f^*E)\cap x)=c_i(E)\cap f_*(x)
\end{equation}
holds in the Chow group $\CH_*(Y)$ for all $x\in \CH^*(X)$, and all $i$ (see \cite[Theorem~3.2(c)]{Fulton}). Let $E$ be a vector bundle on a scheme $Y$, and let $f:X \to Y$ be a flat morphism. Then the pullback formula
\begin{equation}\label{pulback}
f^*(c_i(E)\cap y)=c_i(f^*E)\cap f^*(y)
\end{equation}
holds in the Chow group $\CH_*(X)$ for all $y\in \CH_*(Y)$, and all $i$ (see \cite[Theorem~3.2(d)]{Fulton}). For any exact sequence
\[
0 \to E' \to E \to E'' \to 0, 
\]
of vector bundles, the Whitney sum 
\begin{equation}\label{sumsumsum}
c_l(E) = \Sigma_{i+j=l} c_i(E')c_j(E'')
\end{equation}
holds in $\CH_*(X)$. More generally, for any filtration 
\[
0 = E_0 \subset E_1 \subset \dots E_{r-1} \subset E_{r} =E
\]
of vector bundles, such that for $i = 1, \dots, r$ the $Q_i = E_{i} / E_{i-1}$ are also vector bundles, the Whitney sum 
\begin{equation}\label{sumsumsumfilt}
c_l(E) = \sum_{i_1+i_2+...+i_r=l}\prod_{m=1}^{r}c_{i_{m}}(Q_m)
\end{equation}
holds in $\CH_*(X)$.

The Chern character and Todd class can be defined using Chern roots as follows. Let $X$ be a smooth variety over $k$ of dimension $d$, and $E$ a vector bundle of rank $r$ on $X$. If $P \to X$ is a flat morphism such that the pull-back $\CH^*(X) \to \CH^*(P)$ is injective, and such that the Chern polynomial $c_t(E) = 1 + \sum_{i = 1}^r c_i(E) t^r$ factors as 
\[ c_t(E) = \prod_{i=1}^r(1+a_it) \in \CH^*(P)[[t]] \]
for some $a_1, \dots, a_r \in  \CH^1(P)$ then the $a_1, \dots, a_r$ are called \emph{Chern roots} of $E$. By the splitting principal there always exists such a $P$ (for example, take $P$ to be the flag bundle of $E$), and in fact, for any finite collection $E_1, \dots, E_m$ of vectors bundles, there exists a $P$ which works for all $E_i$ at once. Now, recall that any symmetric polynomial $T(a_1, \dots, a_r) \in \Z[[a_1, \dots, a_r]]^{Sym(r)}$ can be written in a unique way $T = f(\sigma_1, \dots, \sigma_r)$ as a polynomial $f$ in elementary symmetric polynomials $\sigma_i(a_1, \dots, a_r) = \sum_{1 \leq j_1 < \dots < j_i \leq r} a_{j_1} \dots a_{j_i}$. It follows from this that given any such symmetric polynomial $T$, any such $P \to X$ and any Chern roots $a_1, \dots, a_r \in CH^*(P)$ of $E$, the class $T(a_1, \dots, a_r)$ lies in $\CH^*(X) \subseteq \CH^*(P)$.  %
%
%
If a symmetric polynomial $T\in\Z[[a_1, \dots, a_r]]^{Sym(r)}$ is homogeneous, then 
$T(a_1, \dots, a_n)$ is contained in $\CH^{\dim X -\deg T}(X)$.

The Chern character is defined as

\begin{equation} \label{ChQ} 
ch(E)=\Sigma_{n=1}^{r} \exp(a_n)=\Sigma_{n=1}^{r}\Sigma_{i=0}^{\infty}\frac{1}{i!}a_n^i\in \CH^*(X)_{\mathbb{Q}},
\end{equation}

and the Todd class is defined as

\begin{equation} \label{TdQ} 
Td(E)=\prod_{n=1}^r \frac{a_n}{1-\exp(-a_n)}=\prod_{n=1}^r (\Sigma_{i=0}^\infty \frac{(-1)^iB_i}{i!}a_n^i) \in\CH^*(X)_{\mathbb{Q}},
\end{equation}
for any choice of Chern roots $a_1, \dots, a_n$.

\subsection{Study of the Chern character and the Todd class with integral coefficient}\label{23}
In this section, we study the Chern character and the Todd class with integral coefficient. Let $X$ be a smooth variety over $k$ of dimension $d$, and $E$ be a vector bundle of rank $r$ on $X$. Let $a_1,a_2,...,a_r$ be Chern root of $E$. For a commutative ring $R$, we write $R[a_1,..,a_r]^{Sym(r)}_{d}$ for the residue class ring 
\[
R[a_1,..,a_r]^{Sym(r)}_{d}:=R[[a_1,..,a_r]]^{Sym(r)}/ (f:\text{homogeneous polynomial}~|~\deg(f)>d).
\]
Since the natural map of commutative rings
\begin{eqnarray}
    \Z[[a_1,a_2,...,a_r]]^{Sym(r)} &\to & \CH^*(X) \\
\sigma_i(a_1,...,a_r)    & \mapsto & c_i(E) \nonumber
\end{eqnarray}
factors through $\Z[a_1,a_2,...,a_r]^{Sym(r)}_d$:
\[
\Z[[a_1,a_2,...,a_r]]^{Sym(r)} \to \Z[a_1,a_2,...,a_r]^{Sym(r)}_d \to \CH^*(X).
\]
We note that the following diagram of rings 
\[
\xymatrix{
\Z[[a_1,a_2,...,a_r]]^{Sym(r)} \ar[r] \ar@{^{(}-_>}[d] & \Z[a_1,a_2,...,a_r]^{Sym(r)}_d \ar@{^{(}-_>}[d] \ar[r] & \CH^*(X) \ar[d]\\
\Q[[a_1,a_2,...,a_r]]^{Sym(r)} \ar[r] & \Q[a_1,a_2,...,a_r]^{Sym(r)}_d \ar[r]& \CH^*(X)_\Q
}
\]
is commutative where vertical maps are natural base change maps. We note that $\Z[[a_1,a_2,..,a_r]]^{Sym(r)}$ and $\Z[a_1,a_2,..,a_r]^{Sym(r)}_d$ are torsion free $\Z$-modules. For a number $l \geq d$, by the definition of Chern character, if we multiply the class of the Chern character $\ch(E) = \Sigma_{n=1}^{r} \exp(a_n) \in \Q[a_1,a_2,...,a_r]^{Sym(r)}_d$ by $l!$, we can regard $l!\cdot \ch(E)$ as an element of $\Z[a_1,a_2,...,a_r]^{Sym(r)}_d$:
\[
l!\cdot \ch(E) \in \Z[a_1,a_2,...,a_r]^{Sym(r)}_d.
\]
Set 
\[
T_m:=\displaystyle\prod_{p:\text{prime number}} p^{[\frac{m}{p-1}]}
\]
for a number $m \in \N$. By \cite[Lemma~1.7.3]{Hi}, if we multiply the class of the Todd class $Td(E) \in \Q[a_1,a_2,...,a_r]^{Sym(r)}_d$ by $T_l$, we can regard $T_l\cdot Td(E)$ as an element of $\Z[a_1,a_2,...,a_r]^{Sym(r)}_d$:
\[
T_l \cdot Td(E) \in \Z[a_1,a_2,...,a_r]^{Sym(r)}_d,
\]
and also if we multiply the class of the inverse of the Todd class $Td(E)^{-1} \in \Q[a_1,a_2,...,a_r]^{Sym(r)}_d$ by $T_l$, we can regard $T_l\cdot Td(E)^{-1}$ as an element of $\Z[a_1,a_2,...,a_r]^{Sym(r)}_d$:
\[
T_l \cdot Td(E)^{-1} \in \Z[a_1,a_2,...,a_r]^{Sym(r)}_d.
\]
Since the product $(T_l\cdot Td(E) ) \cdot (T_l\cdot Td(E)^{-1}) $ is equal to $T_l^2\cdot 1 \in \Z[a_1,a_2,..,a_r]^{Sym(r)}_d$, we have an equality 
\begin{equation}
    (T_l\cdot Td(E) ) \cdot (T_l\cdot Td(E)^{-1}) =T_l^2\cdot [X]
\end{equation}
in $\CH^*(X)$.

For a flat morphism $f:X \to Y$, a vector bundle $E$ on $Y$ and a number $l \geq \max\{\dim X,\dim Y\}$, by the equation \eqref{pulback} we have
\begin{eqnarray}
l!\cdot \ch(f^*E)&=&f^*\bigl( l!\cdot\ch(E) \bigr)\label{pullbackch} \\ 
T_l\cdot Td(f^*E) &=& f^*\bigl( T_l\cdot Td(E)\bigr)\label{pullbackTd}
\end{eqnarray}
in $\CH^*(X)$, and for a proper morphism $f:X \to Y$ by the projection formula \eqref{projection} we have
\begin{eqnarray}\label{projchch}
f_*\bigl(l!\cdot \ch(f^*E)\cap \alpha\bigr)&=&l!\cdot \ch(E)\cap f_*\alpha\label{projectionch}
\end{eqnarray}
\begin{eqnarray}\label{projTdTd}
f_*\bigl(T_l\cdot Td(f^*E)\cap \alpha\bigr)&=&T_l \cdot Td(E)\cap f_*\alpha
\end{eqnarray}
for any $\alpha \in \CH^*(X)$.

Take $l\geq d$. Let $E_1$ and $E_2$ be vector bundles of rank $r_1$ and $r_2$ on smooth variety $X$, we assume there is an exact sequence of vector bundles on $X$:
\[
0 \to E_1 \to E \to E_2 \to 0.
\]
Let us now prove $\ch(E)=\ch(E_1)+\ch(E_2)$. Set $a_1,a_2,...,a_{r_1}$ to be Chern roots of $E_1$, and $b_1,b_2,...,b_{r_2}$ to be Chern roots of $E_2$. By the Whitney sum \eqref{sumsumsum}, we have
\[
c_t(E) =c_t(E_1)c_t(E_2).
\]
Thus $a_1,a_2,...,a_{r_1},b_1,b_2,...,b_{r_2}$ are Chern roots of $E$.
Now consider the homogeneous polynomial
\[
T_n:= \Sigma_{i=1}^{r_1}a_i^n +  \Sigma_{i=1}^{r_2}b_i^n.
\]
The Chern character $l!\cdot ch(E)$ is equal to $\Sigma_{i=0}^d\frac{l!}{n!}T_n$ in $\CH^*(X)$. By the definition, we have the additive formula in $\CH^*(X)$
\begin{eqnarray}\label{nakanosine}
l!\cdot \ch(E)=\Sigma_{i=0}^d\frac{l!}{n!}T_n&=&\Sigma_{i=0}^d\frac{l!}{n!}(\Sigma_{i=1}^{r_1}a_i^n) +
\Sigma_{i=0}^d\frac{l!}{n!}(\Sigma_{i=1}^{r_2}b_i^n)\\
&=& l! \cdot\ch(E_1) +l! \cdot\ch(E_2).\nonumber
\end{eqnarray}
In the same way, we have the additive formula
\begin{equation}\label{Dentsuusine}
T_l^2\cdot Td(E)=T_l\cdot \prod_{i=1}^{r_1}(\Sigma_{n=0}^l \frac{(-1)^nB_n}{n!}a_i^n) \cdot T_l\cdot \prod_{j=1}^{r_2}(\Sigma_{n=0}^{l} \frac{(-1)^nB_n}{n!}b_j^n)=T_l\cdot Td(E_1) \cdot T_l\cdot Td(E_2)
\end{equation}
in $\CH^*(X)$.
Thanks to the additive formula \eqref{nakanosine}, the map 
\begin{eqnarray*}
\Z[{\textbf{Vect}}(X)] & \to & \CH^*(X)\\
\Sigma m_i E_i & \mapsto &  \Sigma m_i \cdot  l!\cdot ch(E_i)
\end{eqnarray*}
induces a homomorphism of additive groups
\[
\ch:K^0(X) \to \CH^*(X).
\]

\subsection{Chern character of tensor products}\label{24} We set 
\[
\exp^{(l)}(\alpha)=\Sigma_{n=0}^{l}\frac{1}{n!}\alpha^n. 
\] 
It is easy to see that there are polynomials $f_{j}\in R_{l}[a,b]$ satisfying the product rule 
\begin{equation}\label{expexp}
\exp^{(l)}(a)\exp^{(l)}(b)=\exp^{(l)}(a+b) + \Sigma_{j=0}^{l+1} a^jb^{l+1-j}f_{j}.
\end{equation}
Notice that if we consider $R[a,b]$ as a graded ring with $a$ and $b$ of degree one, then the error term $\Sigma_{j=0}^{l+1} a^jb^{l+1-j}f_{j}$ has degree $\geq l+1$.

For a vector bundle $F$ of rank $e$, and let $c_1,..,c_e$ be  Chern roots of $F$, then 
 \[ l!\cdot \ch(F)=l!\cdot \Sigma_{n=1}^e\exp^{(l)}(c_n) \]
 in $\Z[c_1,c_2,..,c_e]^{Sym(e)}_d$.
 
\begin{lemma}
Let $X$ be a smooth variety of dimension $d$ over $k$. For a vector bundle $E$ of rank $r$, and a vector bundle $E'$ of rank $r'$. Let $a_1,...,a_r$ be  Chern roots of $E$ and $b_1,...,b_{r'}$ be Chern roots of $E'$. Then 
\[
a_i+b_j, \text{ ~~}\:\:\:~~ 1\leq i \leq r \:\:\: \text{ ~~ } 1\leq j \leq r' .
\] 
are Chern roots of $E \otimes E'$.
\end{lemma}

\begin{proof}
In the case $r=1$ and $r'=1$, since $c_1(E\otimes E')=c_1(E) + c_1(E')$ in $\CH^*(X)$, we have
\[
c_t(E\otimes E')=1+c_1(E)t+c_1(E')t=1+a_1t+b_1t = 1 + (a_1 + b_1)t.
\]

By the splitting principal there is a flat projective map $f:P \to X$ and vector bundles $E_i$ and $E'_{i'}$ on $P$ giving filtrations of vector bundles
\[
0 \subset E_1 \subset \dots \subset E_{r-1}\subset E_{r} = f^*E
\] 
and 
\[
0 \subset E'_1 \subset \dots \subset E'_{r'-1}\subset E'_{r'} = f^*E'
\]
such that $L_{i}=E_i/E_{i-1}$ and $L'_{i}=E'_{i}/E'_{i-1}$ are line bundles for any $i$. Thanks to these filtrations, there is a bi-filtration of $f^*(E\otimes E')$ such that the each of $(i,j)$-graded piece is isomorphic to the line bundle $L_i\otimes L'_{j}$ on $P$. By the definition of Chern roots, we know $c_1(L_i)=a_i$ and $c_1(L'_{j})=b_{j}$, thus we have $c_t(L_i\otimes L'_{j})=1+(a_i+b_j)t$. By the Whitney sum~\ref{sumsumsumfilt} we have
\[
c_t(E\otimes E')= \prod_{0\leq i\leq r}\prod_{0\leq j \leq r'} c_t(L_i \otimes L'_j)=\prod_{0\leq i\leq r}\prod_{0\leq j \leq r'} (1+(a_i+b_j)t) \in \CH^*(P)[[t]].
\]
The proof is then achieved.
\end{proof}
By the equation \eqref{expexp}, if $l \geq d$ then We have that
\begin{equation}\label{otimes}
l!^2 \cdot \ch(E\otimes E')
=l!^2 \cdot \sum_{i=1}^r \sum_{j=1}^{r'}\exp^{(l)}(a_i+b_j)
=\left (l! \cdot\sum_{i=1}^r  \exp^{(l)}(a_i) \right ) \left (l! \cdot\sum_{j=1}^{r'}\exp^{(l)}(b_j) \right )
=l! \cdot \ch(E)\cdot l! \cdot\ch(E')
\end{equation}
in $\CH^*(X)$.

\subsection{Chern character of exterior powers}
Let $a_1,...,a_r$ be Chern roots of a vector bundle $E$ on a smooth variety $X$ of dimension $d$ Then the Chern polynomial of the exterior powers $\bigwedge^p E$ satisfies the following:
\[
c_t(\bigwedge^p E) = \displaystyle\prod_{i_1<i_2<\cdots < i_p}\bigl(1+ (a_{i_1}+a_{i_2}+\cdots + a_{i_p})t\bigr).
\]

\begin{example}(\cite[Example~3.2.5]{Fulton})\label{exexex}
Let $a_1,...,a_r$ be Chern roots of a vector bundle $E$ on smooth variety $X$ of dimension $d$. We will show the equation
\begin{equation}\label{exFulton}
T_l \cdot \Sigma_{p=0}^r(-1)^p \ch(\bigwedge^p E^{\vee})= T_l \cdot Td(E)^{-1}c_d(E)
\end{equation}
in $\CH^*(X)$.  Since $l!$ divides $T_l$, both sides of the equation have integral coefficients. Thus it is enough to show that the equation holds in $\Q[a_1,a_2,..,a_r]^{Sym(r)}_d$. Since the Chern roots of $\bigwedge^p E^{\vee}$ are $\{-a_{i_1}-a_{i_2}-\cdots-a_{i_p}| i_1<i_2< \cdots < i_p \}$, for $l\geq d$  we have 
\begin{equation}\label{kaiji}
\Sigma_{p=0}^r(-1)^p \ch(\bigwedge^p E^{\vee})= \cdot \Sigma_{p=0}^r(-1)^p {\sum_{i_1<i_2<\cdots <i_p}}\exp^{(l)}(-a_{i_1}-a_{i_2}-\cdots-a_{i_p})
\end{equation}
in $\Q[a_1,a_2,..,a_r]^{Sym(r)}_d$, where $\exp^{(l)}(\alpha)=\Sigma_{n=0}^{l}\frac{1}{n!}\alpha^n$. By the product rule of $\exp^{(l)}$ (see Equation \eqref{expexp}) we have
\begin{equation}\label{akagi}
\Sigma_{p=0}^r(-1)^p {\sum_{1\leq i_1<i_2<\cdots <i_p\leq r}}\exp^{(l)}(-a_{i_1}-a_{i_2}-\cdots-a_{i_p})= \displaystyle\prod_{i=1}^{r}(1-\exp^{(l)}(-a_i))
\end{equation}
in $\Q[a_1,a_2,..,a_r]^{Sym(r)}_d$.
Recall that in $\Q[[x]]$ we have 
\[ \frac{x}{1-\exp(-x)} = \Sigma_{n=0}^\infty \frac{(-1)^nB_n}{n!}{x}^n. \]
From this we obtain
\begin{align*}
x &= (1-\exp(-x))\left (\Sigma_{n=0}^\infty \frac{(-1)^nB_n}{n!}{x}^n \right ) \\
&= \biggl (1-\exp^{(l)}(-x) + f(x) \biggr )\left ( \Sigma_{n=0}^l \frac{(-1)^nB_n}{n!}{x}^n + g(x) \right) \\
&= \biggl (1-\exp^{(l)}(-x)\biggr )\left ( \Sigma_{n=0}^l \frac{(-1)^nB_n}{n!}{x}^n\right) + h(x) 
\end{align*}
where $f(x)$, $g(x)$ and $h(x)$ are in $(x^{l+1})\Q[[x]]$. This equality induces the following equality of symmetric polynomials 
\[
\displaystyle\prod_{i=1}^r \bigl(1-\exp^{(l)}(-a_i)\bigr)\bigl(\Sigma_{n=0}^l \frac{(-1)^nB_n}{n!}a_i^n\bigr)= a_1a_2\cdots a_r + Q_{d+1} + Q_{d+2} +\cdots Q_{2lr}=c_d(E)+ Q_{d+1} + Q_{d+2} +\cdots Q_{2lr}
\]
where the $Q_{d+i}\in \Q[a_1,a_2,...,a_r]$ are symmetric polynomials of degree $d+i$. By the equation \eqref{kaiji} and the equation \eqref{akagi}, we know \[\prod_{i=1}^r \bigl(1-\exp^{(l)}(-a_i)\bigr)=\Sigma_{p=0}^r(-1)^p \ch(\bigwedge^p E^{\vee}).\]
By the definition, we know $Td(E)=\prod_{i=1}^r\Sigma_{n=0}^l \frac{(-1)^nB_n}{n!}a_i^n$. On the other hand, we know $Q_{d+i}=0$ in $\CH^*(X)_{R_{l+1}}$ thus we have the equation
\[
\Sigma_{p=0}^r(-1)^p \ch(\bigwedge^p E^{\vee})=Td(E)^{-1}c_d(E)
\]
in $\Q[a_1,a_2,...,a_r]^{Sym(r)}_d$
\end{example}

\subsection{Integral Grothendieck Riemann-Roch for sections of projective bundles}\label{26}

Let $X$ be a smooth variety of dimension $d$ over $k$, and let $N$ be a vector bundle of rank $r$ on $X$. Take $l \geq r+d$. Let $f$ be the closed immersion $X \hookrightarrow \PP_X(N\oplus 1)$ which is the composition of the zero section $X \to N$ and canonical open immersion $N \to \PP_X(N\oplus 1)$. Let $p$ be the natural projection $\PP_X(N\oplus 1) \to X$. 
\[
\xymatrix{
X \ar[r]^-{f} & \PP_X(N \oplus 1) \ar[d]^{p} \\
& X}
\]
In this setting we prove the following.
\begin{prop}\label{GRRprojsection}
Suppose $l \geq r+d$. For any vector bundle $E$ on $X$, the equation
\begin{equation}\label{GRRimmersion}
T_l^2\cdot (l!)^2 \cdot Td(T_{\PP_X(N\oplus 1)})\cdot \ch(f_*[E]) = f_*\bigl( T_l^2\cdot (l!)^2 \cdot Td(T_X)\ch(E)\bigr)
\end{equation}
holds in $\CH^*(\PP_X(N\oplus 1))$.
\end{prop}
\begin{proof}
Let $\sQ$ be the universal quotient bundle of $p$ on $\PP_X(N\oplus 1)$ i.e.
\[
\sQ:=p^*(N\oplus 1)/(L^{\text{taut}})
\]
where $L^{\text{taut}}$ is the tautological bundle of $p$. Let $s$ be the section of $Q$ determined by the projection of the trivial factor in $p^*(N\oplus 1)$ to $\sQ$. In this setting, for any vector bundle $E$ on $X$, the Koszul complex induces the following resolution of $f_*E$:
\[
0 \to \bigwedge^d \sQ^\vee \otimes p^*E \to ... \to   \sQ^\vee \otimes p^*E \to p^*E \to f_*E \to 0
\]
see the details in \cite[\S 15.1]{Fulton}. By Whitney sum formula \eqref{nakanosine} and the tensor product formula \eqref{otimes}, we have
\begin{equation}\label{lightkun}
(l!)^2\cdot \ch(f_*E)=\sum_{p=0}^d(-1)^p l!\cdot\ch\left ( \wedge^p \sQ^{\vee} \right ) \cdot l!\cdot \ch(p^*E)
\end{equation}
in $\CH^*(\PP_X(N\oplus 1))$. By Example~\ref{exexex}, we have
\begin{equation}\label{cleaner}
T_l\cdot \sum_{p=0}^d(-1)^p\ch(\bigwedge^p \sQ^{\vee})=T_l\cdot c_d(\sQ)  \cdot  Td(\sQ)^{-1}
\end{equation}
in $\CH^*(\PP_X(N\oplus 1))$. Since $l!$ divides $T_l$, combine the equations \eqref{lightkun} and \eqref{cleaner}, we obtain the equality
\begin{equation}\label{batman}
   l!\cdot T_l \cdot  \ch(f_*E) = l!\cdot T_l \cdot c_d(\sQ)  \cdot  Td(\sQ)^{-1}  \cdot \ch(p^*E)
\end{equation}
in $\CH^*(\PP_X(N\oplus 1))$. By \cite[Proposition~14.1]{Fulton} we know
\begin{equation}\label{granma}
f_*f^*\alpha = c_d(\sQ) \alpha 
\end{equation}
in $\CH^*(\PP_X(N\oplus 1))$ for any $\alpha \in \CH^*(\PP_X(N\oplus 1))$, where $f^*\alpha$ is the pull-back along regular embedding $f$. Thus we have equalities
\begin{equation}\label{Step1}
l!\cdot T_l \cdot \ch(f_*[E]) 
\stackrel{\eqref{batman}}{=} 
l!\cdot T_l \cdot c_d(\sQ) \cdot Td(\sQ)^{-1} \cdot \ch(p^*E)
\stackrel{\eqref{granma}}{=} 
f_*\left(f^*(T_l \cdot Td(\sQ)^{-1})\cdot f^*(l! \cdot\ch(p^*E)) \right).
\end{equation}

Now since $N$ is the normal bundle of $f$, and the varieties $X$ and $\PP_X(N\oplus 1)$ are smooth over $k$,  there is an exact sequence of a vector bundle on $X$
\begin{equation}\label{kimura}
0 \to  T_X \to f^*T_{\PP_X(N\oplus 1)} \to N \to 0.
\end{equation}
So by the equation \eqref{pullbackTd} and the additive formula of integral Todd class \eqref{Dentsuusine}, \cite[Proposition 6.3]{Fulton}, we have
\begin{equation}\label{ToddStep2}
T_l^2 \cdot Td(T_X)=T_l\cdot Td(N)^{-1}\cdot f^*( T_l \cdot Td(T_{\PP_X(N\oplus 1)}))
\end{equation}
in $\CH^*(X)$. Now we know $f^*p^*E=E$ and $f^*\sQ=N$, so since $l!\cdot \ch$ and $T_l \cdot Td$ commute with pull-backs along a regular embedding (see \cite[Proposition 6.3]{Fulton}), if we multiply \eqref{Step1} by $T_l \cdot Td(T_{\PP_X(N\oplus 1)})$ we have the following
\begin{eqnarray} \label{fastball}
l!\cdot T_l\cdot\ch(f_*[E]) \cdot  T_l \cdot Td(T_{\PP_X(N\oplus 1)})
&\stackrel{\eqref{Step1}}{=}&f_*\left(f^*( T_l\cdot Td(\sQ)^{-1})\cdot f^*(l!\cdot \ch(p^*E)) \right) \cdot  T_l \cdot Td(T_{\PP_X(N\oplus 1)}) \nonumber \\
&\stackrel{\textrm{proj.form.}}{=}&f_*\left(T_l\cdot Td(N)^{-1}\cdot l! \cdot \ch(E) \cdot f^*(T_l \cdot Td(Td_{\PP_X(N\oplus 1)})) \right)\nonumber\\
&\stackrel{\eqref{ToddStep2}}{=}& f_*\bigl(T_l^2\cdot Td(T_X) \cdot l!\cdot  \ch(E)\bigr)\nonumber
\end{eqnarray}
in $\CH^*(\PP_X(N\oplus 1))$. Thus we obtain the desired equality
\begin{equation*}
T_l^2\cdot (l!)^2 \cdot Td(T_{\PP_X(N\oplus 1)}) \cdot \ch(f_*[E]) = f_*\bigl(T_l^2\cdot (l!)^2 \cdot Td(T_X) \cdot \ch(E)\bigr)
\end{equation*}
in $\CH^*(\PP_X(N\oplus 1))$. 
\end{proof}

\subsection{Integral Grothendieck Riemann-Roch for projections}\label{GRRproj}
For smooth projective varieties $X$ and $Y$ over $k$, we consider the map 
\begin{eqnarray*}
K_0(X)\otimes K_0(Y) &\to& K_0(X\times Y)\\
  ~[E]\otimes [G] & \mapsto & [p^*E \otimes q^*G] 
\end{eqnarray*}
where $p$ is the projection $X\times Y \to X$ and $q$ is the projection $X \times Y \to Y$, and the map
\begin{eqnarray*}
\CH^*(X) \otimes_\Z \CH^*(Y) &\to& \CH^*(X\times Y)\\
Z \otimes W &\mapsto & p^* Z \cdot q^*W
\end{eqnarray*}
Consider the composition 
\[
\psi:K_0(X)\otimes K_0(Y) \to K_0(X\times Y) \overset{p_*}{\to} K_0(X)
\]
and the composition
\[
\phi:\CH^*(X) \otimes_\Z \CH^*(Y) \to \CH^*(X\times Y) \overset{p_*}{\to} \CH^*(X).
\]
For a vector bundle $E$ on $X$ and a vector bundle $G$ on $Y$, by the projection formula, we have 
\begin{equation}\label{pingpongfree}
\psi(E\otimes G)=p_*\left( p^*E\otimes q^*G \right)= E\otimes p_*q^*G
\end{equation}
in $K_0(X)$. Let us study $p_*q^*G$.
\begin{lemma}\label{lemlem}
Consider the Cartesian diagram  of Noetherian schemes
\[
\xymatrix{
N'\ar[r]^{h'} \ar[d]_{g'} & N \ar[d]^{g}\\
M' \ar[r]_{h}& M}
\]
where $g$ is proper. If $h$ is flat, then there is an equation
\[
h^*g_*\alpha ={g'}_*{h'}^*\alpha
\]
in $K_0(M')$ for any $\alpha \in K_0(N)$.
\end{lemma}
\begin{proof}
It is enough to show that
\[
h^*R^ig_*\sF\simeq R^i{g'}_*{h'}^*\sF
\]
for any $i\in\Z$ and any coherent sheaf $\sF$ on $N$. This isomorphism holds in the case $h$ is flat.
\end{proof}
We denote the structure map of $X$ by $f$, the structure map of $Y$ by $g$.
\[
\xymatrix{
X\times Y \ar[r]^-{q} \ar[d]_-{p} & Y \ar[d]^-{g}\\
X \ar[r]_-{f}& \Spec k}
\]
For a vector bundle $G$ on $Y$, by the definition of push-forward for the Grothendieck group, we have
\begin{equation}\label{pingpong}
g_*G=\left(\Sigma_{i=0}^\infty(-1)^i\dim_k H^i(Y,G)\right)  \cdot [\sO_{\Spec k}] \in K_0(\Spec k).
\end{equation}
By Lemma~\ref{lemlem} and the equation~\eqref{pingpong}, we have
\begin{equation}\label{pingpong2}
p_*q^* G =f^*g_*G=\left(\Sigma_{i=0}^\infty(-1)^i\dim_k H^i(Y,G)\right)\cdot[\sO_X]
\end{equation}
in $K_0(X)$. Combine the equations \eqref{pingpongfree} and \eqref{pingpong2}, we have
\begin{equation}\label{GOODPING}
    \psi(E\otimes G)\overset{\eqref{pingpongfree}}{=}E\otimes p_*q^*G \overset{\eqref{pingpong2}}{=} \left(\Sigma_{i=0}^\infty(-1)^i\dim_k H^i(Y,G)\right)\cdot E
\end{equation}
in $K_0(X)$.

Since $\ch$ and $Td$ commute with pull-backs,  \eqref{pullbackch} and \eqref{pullbackTd} respectively, and $T_{X \times Y} \cong p^*T_X \oplus q^*T_Y$ and the additive formula for Todd classes \eqref{Dentsuusine} and tensor product formula for Chern character \eqref{otimes} we have \footnotesize
\begin{eqnarray}
p^*\biggl (l!\cdot \ch(E)\cdot T_l \cdot Td(T_X) \biggr )\cdot q^* \biggl (l!\cdot \ch(G)\cdot T_l\cdot Td(T_Y)\biggr )&=& l!\cdot p^*\ch(E)\cdot T_l\cdot p^*Td(T_X)\cdot l!\cdot q^*\ch(G)\cdot T_l\cdot q^*Td(T_Y)\nonumber \\
&=& (l!)^2\cdot \ch\left(p^*E\otimes q^*G  \right) \cdot T_l^2\cdot Td(p^*T_X\oplus q^*T_Y) \nonumber \\
&=&  (l!)^2\cdot \ch\left(p^*E\otimes q^*G  \right) \cdot T_l^2\cdot Td(T_{X\times Y}) \nonumber 
\end{eqnarray} \normalsize
in $\CH^*(X\times Y)$ for any vector bundle $E$ on $X$ and any vector bundle $G$ on $Y$. Thus the diagram \footnotesize
\begin{equation}\label{first}\xymatrix{
K_0(X)\otimes K_0(Y) \ar[rrrrrrr]^-{l!\cdot \ch(-)\cdot T_l \cdot Td(T_X)\otimes l! \cdot \ch(-)\cdot T_l \cdot Td(T_Y)} \ar[d]_{p^*(-)\otimes q^*(-)}&&&& &&& \CH^*(X) \otimes_\Z \CH^*(Y) \ar[d]^{p^*(-)\cdot q^*(-) }\\
K_0(X\times Y)\ar[rrrrrrr]^{(l!)^2 \cdot \ch(-)\cdot T_l^2\cdot Td(T_{X\times Y})}&&&& &&& \CH^*(X\times Y)\\
}
\end{equation} \normalsize
commutes. 

For a cycle $Z$ on $X$, a cycle $W$ on $Y$, by the projection formula of cycles, we have
\begin{equation}\label{pandafree}
    \phi(Z\otimes W)=p_*\left(p^*Z\cdot q^*W\right) = Z \cdot p_*q^*W.
\end{equation}
We denote by $W_0$ the $0$-dimensional part of $W$. Then by the definition of the flat pull-back of cycles and proper push-forward of cycles, we have
\begin{equation}\label{panda}
   p_*q^* W = (\deg W_0)\cdot [X]
\end{equation}
By the equations~\eqref{pandafree}, we have
\begin{equation}\label{GOODPANDA}
\phi(Z\otimes W)=(\deg W_0)\cdot [Z].
\end{equation}

\begin{prop}\label{GRRprojection}
If $Y=\PP^m$ and $l \geq \dim X + m$, then the equation
\[
(l!)^2 \cdot T_l^2 \cdot \ch(p_*x)\cdot  Td(T_X)=  p_*\bigl((l!)^2 \cdot T_l^2 \cdot \ch(x)\cdot Td(T_{X\times Y})\bigr)
\]
holds in $\CH^*(X)$ for any $x\in K_0(X\times Y)$.
\end{prop}
\begin{proof}
Consider $Y=\PP^m$. The Grothendieck group $K_0(Y)$ is generated by $[\sO(n)],n\geq 0$ (see \cite[Example~15.1.1]{Fulton}). For a vector bundle $E$ on $X$, by \eqref{GOODPING} we have
\begin{equation}\label{OSAMU}
l!\cdot \ch(\psi([E]\otimes[\sO(n)]))\cdot T_l \cdot  Td(T_X)=\chi(\PP^m,\sO(n))\cdot l! \cdot \ch([ E])\cdot T_l \cdot Td(T_X).
\end{equation}
Also by the equation \eqref{GOODPANDA}, we have
\begin{align}\label{AKIYAMA}
\phi\bigl(l!\cdot \ch(E)\cdot T_l \cdot  Td(T_X)
&\otimes l! \cdot  \ch(\sO(n))\cdot T_l \cdot  Td(T_{\PP^m})\bigr) \\
&=\deg\bigl((l!\cdot  \ch(\sO(n))\cdot T_l \cdot  Td(T_{\PP^m}))_0\bigr)\cdot l! \cdot \ch([ E])\cdot T_l \cdot  Td(T_X).\nonumber
\end{align}
By Hirzebruch Riemann-Roch theorem we have
\begin{equation}\label{TEZUKA2}
\chi(\PP^m,\sO(n))= \deg\bigl(( \ch(\sO(n))\cdot Td(T_{\PP^m}))_0\bigr).
\end{equation}
By equations \eqref{OSAMU}, \eqref{AKIYAMA} and \eqref{TEZUKA2} we obtain that 
\[
\ch(\psi([E]\otimes[\sO(n)]))\cdot Td(T_X)=\phi\bigl(\ch(E)\cdot Td(T_X)\otimes \ch(\sO(n))\cdot Td(T_{\PP^m})\bigr)
\]
i.e. in the following diagram \footnotesize 
\[\xymatrix{
K_0(X)\otimes K_0(Y) \ar[rrrrrrr]^-{l!\cdot \ch(-)\cdot T_l \cdot Td(T_X)\otimes l! \cdot  \ch(-)\cdot T_l \cdot Td(T_Y)} \ar@/_30pt/[dd]_{\psi} \ar[d]_{}&&&&&& & \CH^*(X) \otimes_\Z \CH^*(Y) \ar[d]  \ar@/^40pt/[dd]^{\phi} \\
K_0(X\times Y)\ar[d]_{p_*} \ar[rrrrrrr]^{(l!)^2 \cdot \ch(-)\cdot T_l^2 \cdot Td(T_{X\times Y})}&&& && &&\CH^*(X\times Y)\ar[d]^{p_*}\\
K_0(X) \ar[rrrrrrr]_{(l!)^2 \cdot \ch(-) \cdot T_l^2 \cdot Td(T_X)}&&&&&& &\CH^*(X)
}
\]\normalsize
the outer square and the upper square are commutative (see \eqref{first}). Since we take $Y=\PP^m$ the map $K_0(X) \otimes K_0(Y) \to K_0(X\times Y)$ is surjective (see \cite[example~15.1.1]{Fulton}), so the lower square in the diagram also commutes. Hence we obtain that
\[
(l!)^2\cdot T_l^2 \cdot \ch(p_*x)\cdot Td(T_X)= p_*\bigl((l!)^2 \cdot T_l^2 \cdot \ch(x) \cdot Td(T_{X\times Y})\bigr)
\]
in $\CH^*(X)$.\end{proof}

\subsection{Integral Grothendieck Riemann-Roch for closed embeddings}\label{28}
Let $f:X \to Y$ be a closed immersion of quasi-projective varieties over field $k$. We denote the normal cone of $f$ by $C_XY$, and the zero-section $X \to C_XY$ by $s$. In this setting, there is a quadruple $(M,	\varrho,q,F)$ (see \cite[Chapter~5]{Fulton}):
\begin{equation}\label{HOIMI}
\xymatrix{
X \times \PP^1 \ar@<-0.3ex>@{^{(}->}[rr]^-{F} \ar[rd]_-{pr_2} & & M \ar[ld]^{\varrho} \ar[rd]^{q}& \\
& \PP^1=\Proj k[t_0,t_1] & & Y
}
\end{equation}
satisfying followings
\begin{itemize}
\item $M$ is the blowing up of $Y\times \PP^1$ along $X\times\{\infty\}$
    \item $\varrho$ is flat.
    \item $F$ is a closed immersion
    \item Over $\PP^1 \backslash \{\infty\}=\AA^1$, $\varrho^{-1}(\AA^1)=Y\times \AA^1$ and the closed embedding $\varphi$ is the trivial embedding $X\times\AA^1 \to Y \times\AA^1$.
    \item Over $\{\infty\}$, $\varphi$ is $s:X \to C_XY$.
    \item The composition of maps $p \circ \varphi$ is equal to the composition of maps $X\times\PP^1 \overset{pr_1}{\to} X \overset{i}{\to} Y$. 
    \item The divisor $M_{\infty}=\varrho^{-1}(\infty)$ is the sum of two effective Cartier divisors:
    \[
    M_{\infty}= \PP_X(C_XY \oplus 1) + \tilde{Y}
    \]
    where $\tilde{Y}$ is the blowing up of $Y$ along $X$.
\end{itemize}

\begin{example}(\cite[Example~5.1.1]{Fulton})\label{example2.5}
Assume $Y$ is quasi-projective variety over a field $k$. Let $i:Y \to M$ be the closed embedding of $Y$ at $t_0/t_1=0$, and $j$ (resp. $k$) is the canonical embedding of $\PP_X(C_XY \oplus 1)$ (resp. $\tilde{Y}$) in $M$ over $t_1/t_0=0$. Then 
\[
i_* [Y] =j_*[\PP_X(C_XY \oplus 1)] + k_*[\tilde{Y}] 
\]
in $\CH^*(M)$, where $[-]$ indicates the cycle associated to a closed subvariety.
\end{example}

We trace the discussion in \cite[15.2]{Fulton}. Let $f:X \to Y$ be a closed embedding of smooth projective varieties, $d_X$ be the dimension of $X$, and $d_Y$ be the dimension of $Y$. Take $l \geq d_Y$. We prove the following.
\begin{prop}\label{GRRembedding}
Suppose $l \geq d_Y$. Then for any vector bundle $E$ on $X$ there is an equality
\begin{equation}\label{sanban}
   (l!)^2\cdot T_l^2\cdot  \ch(f_*E) Td(T_Y)=(l!)^2\cdot T_l^2\cdot f_*\bigl(\ch(E)Td(T_X)\bigr)
\end{equation}
in $\CH^*(Y)$.
\end{prop}
\begin{proof}
Consider the quadruple $(M,\varrho,p,F)$ (see \eqref{HOIMI}). There is a commutative diagram
\[
\xymatrix{
 & X \ar[d]_{i_\infty} \ar[r]^-{\overline{f}} \ar[ld]_{id}& \PP_X(N\oplus 1) + \tilde{Y} \ar[d]_{k+ g} \ar@{=}[r] & M^\infty \ar[r] \ar[ld]^{j_{\infty}}& \{\infty\}\ar[d]\\
X & X\times \PP^1 \ar[l]^{p} \ar[r]_{F} & M \ar[rr]^{\varrho} & & \PP^1\\
 & X \ar[u]^{i_0} \ar[ul]^{id} \ar[r]_{f} & Y \ar[u]_{j_0} \ar[rr]& & \{0\} \ar[u]}
\]
where $N$ is the normal cone of $f$, and $p$ is the natural projection, $i_{\infty}$ (resp. $i_0$) is $\{\infty\}$-section (resp. $\{0\}$-section), $j_0$ is the natural embedding $\varrho^{-1}(\{0\})=Y \hookrightarrow M$ and $j_\infty$ is the natural embedding $\varrho^{-1}(\{\infty\})=M^\infty\hookrightarrow M$. We denote the natural embedding $\tilde{Y} \hookrightarrow M$ by $g$. We note that $\dim \PP_X(N \oplus 1)=d_Y$, $\dim \tilde{Y} =d_Y$ and $\dim M=d_Y+1$, every square in the diagram is Cartesian. Let $E$ be a vector bundle on $X$. Let $\tilde{E}:=p^*E$, and choose a resolution $G_{\bullet}$ of $F_*(\tilde{E})$ on $M$:
\begin{equation}\label{HUUIN}
0 \to G_n \to G_{n-1} \to ... \to G_0 \to F_*(\tilde{E}) \to 0.
\end{equation}
We note $\varrho \circ F:X\times \PP^1 \to \PP^1$ and $\varrho:M\to \PP^1$ are flt, and since $\tilde{E}$ is a vector bundle on $X\times \PP^1$ thus $F_*\tilde{E}$ is flat over $\PP^1$. We know the exact sequence \eqref{HUUIN} consists of coherent sheaves which are flat over $\PP^1$. Thus the restriction of the exact sequence \eqref{HUUIN} to the fibers $j_0:Y\hookrightarrow M$ and $j_\infty:M^\infty \hookrightarrow Y$ remain exact, we know $j_0^*G_{\bullet}$ is a resolution of $j_0^*F_*(\tilde{E})$ and 
We know $j_0^*G_{\bullet}$ is a resolution of $j_0^*F_*(\tilde{E})$. 

Since we have $\tilde{E}=p^*E$, $p\circ i_0=id_X$, we obtain an equality of coherent sheaves:
\begin{equation}
j_0^*F_*\tilde{E}=f_*i_0^*\tilde{E}=f_*i_0^*p^*E=f_*E.
\end{equation}
Since we know $j_0^*G_{\bullet}$ is a resolution of $j_0^*F_*(\tilde{E})$, thus $j_0^*G_{\bullet}$ is the resolution of $f_*E$. Also we know $\tilde{E}=p^*E$ and $p\circ i_{\infty}=id_X$, we have an equality of coherent sheaves:
\begin{equation}
j_{\infty}^*F_*\tilde{E}=\ol{f}_*i_{\infty}^*\tilde{E}=\ol{f}_*i_{\infty}^*p^*E=\ol{f}_*E.
\end{equation}
Since we know $j_{\infty}^*G_{\bullet}$ is the resolution of $j_{\infty}^*F_*(\tilde{E})$, thus $j_\infty^*G_{\bullet}$ is the resolution of $\overline{f}_*E$.

Since $\overline{f}(X)$ is disjoint with $\tilde{Y}$, thus $k^*G_{\bullet}$ is the resolution of $\overline{f}_*(E)$ on $\PP_X(N\oplus 1)$, and $g^*G_{\bullet}$ is acyclic. By the projection formula \eqref{projectionch} we have
\begin{equation}\label{TAKADA}
{j_0}_*\bigl(l!\cdot \ch(f_*E)\bigr)={j_0}_*\bigl(l!\cdot \ch(j_0^*G_{\bullet})\bigr) \overset{\text{proj.form}}{=}l!\cdot \ch(G_{\bullet})\cap {j_0}_*[Y].
\end{equation}
By example~\ref{example2.5}, we know
\begin{equation}\label{TAKADA2}
l!\cdot \ch(G_{\bullet})\cap {j_0}_*[Y]=l!\cdot \ch(G_{\bullet})\cap \bigl(k_*[\PP_X(N\oplus 1)] + g_*[\tilde{Y}]\bigr).\end{equation}
Continuing, with the projection formula \eqref{projectionch} for $k$ and $g$, since we know $g^*G_{\bullet}$ is acyclic, we have
\begin{equation}\label{TAKADA3}
l!\cdot \ch(G_{\bullet})\cap\bigl(k_*[\PP_X(N\oplus 1)] + g_*[\tilde{Y}]\bigr)\overset{\text{proj.form}}{=}{k}_*\bigl(l!\cdot\ch(k^*G_{\bullet})\bigr)+{g}_*\bigl(l!\cdot\ch(g^*G_{\bullet})\bigr)={k}_*\bigl(l!\cdot\ch(\overline{f}_*E)\bigr).
\end{equation}
Let $\sQ$ be the universal quotient bundle of the projection $\PP_X(N\oplus 1) \to X$ on $\PP_X(N\oplus 1)$. Since we assume $l \geq d_Y = \dim \PP_X(N\oplus 1) =d_X + \text{rank} N $, thus by  by the equation \eqref{Step1} and $\ol{f}^*Q=N$, we have
\begin{eqnarray}\label{TAIHO}
l! \cdot T_l \cdot \ch(\overline{f}_*E) 
\overset{\eqref{ToddStep2}}{=}&\overline{f}_*\bigl(l!\cdot T_l \cdot Td(N)^{-1}\ch(E)  \bigr)
\end{eqnarray}
in $\CH^*(\PP_X(N\oplus 1))$. Now we have the following equations:
\begin{eqnarray}\label{Oniisan}
{j_0}_*\bigl(l!\cdot T_l \cdot \ch(f_*E)\bigr)&\overset{\eqref{TAKADA}}{=}& l!\cdot T_l \cdot \ch(G_{\bullet})\cap {j_0}_*[Y]\\
&\overset{\eqref{TAKADA2}}{=}&l!\cdot T_l \cdot \ch(G_{\bullet})\cap \bigl(k_*[\PP_X(N\oplus 1)] + g_*[\tilde{Y}]\bigr)\nonumber \\
&\overset{\eqref{TAKADA3}}{=}&l!\cdot T_l \cdot {k}_*\bigl(\ch(\overline{f}_*E)\bigr) \nonumber \\
&\overset{\eqref{TAIHO}}{=}& l!\cdot T_l \cdot k_*\overline{f}_*\bigl( Td(N)^{-1}\ch(E)  \bigr)\nonumber
\end{eqnarray}
in $\CH^*(M)$. Let $q$ be the composition of the blowing down $M \to Y \times \PP^1$ followed by the projection to $Y$. Since $q\circ j_0 = id_{Y}$ and $q \circ k \circ \overline{f}=f$, applying $q_*$ to \eqref{Oniisan}, we deduce
\begin{equation}\label{matsu}
l!\cdot T_l \cdot \ch(f_*E)=f_*\bigl( l!\cdot T_l \cdot Td(N)^{-1}\cdot \ch(E)  \bigr).
\end{equation}
Since $N$ is the normal cone of $f$, by the additive formula of integral Todd class, we have 
\[
T_l^2 \cdot Td(T_X)=T_l\cdot Td(N)^{-1} \cdot T_l \cdot  f^*Td(T_Y),
\]
and by the projection formula, w obtain 
\[
T_l^2 \cdot f_*Td(T_X)=T_l\cdot f_*Td(N)^{-1} \cdot T_l \cdot Td(T_Y)
\]
    Thanks to this equation, if we multiply the equation \eqref{matsu} by $T_l \cdot Td(T_Y)$ then we obtain
\[
l!\cdot T_l^2 \cdot \ch(f_*E)\cdot Td(T_Y)=f_*\bigl( l!\cdot T_l^2 \cdot Td(N)^{-1}\ch(E)  \bigr)\cdot T_l\cdot Td(T_Y)=f_*\bigl( l!\cdot T_l^2\cdot \ch(E) \cdot Td(T_X)  \bigr).  \]
This is the what we want.
\end{proof}

\subsection{Proof of Integral Grothendieck Riemann-Roch}\label{29}
\begin{thm}[Integral Grothendieck Riemann-Roch] \label{GRRGRR}\label{IUTT}
Let $f:X \to Y$ be a projective morphism of smooth projective varieties over a field $k$. We assume $X,Y \in \Sm\Proj^{\leq d}_{(e)}(k)$. Then for all $x\in K_0(X)$
\[
f_*\bigl((l!)^2\cdot T_l^2 \cdot \ch(x)Td(T_X)\bigr) = (l!)^2\cdot T_l^2 \cdot \ch(f_*x)Td(T_Y)
\]
in $\CH^*(Y)$ for any $l\geq d+e$.
\end{thm}

\begin{proof}
We fix a closed immersion $i:X \to \PP^e_k$. Then the projective map $f$ induces a closed embedding
\[
\ol{f}:X \overset{(f,i)}{\to} Y \times \PP^e
\]
and a projection 
\[
p:Y \times \PP^e \to Y.
\]
such that $f=p \circ \ol{f}$. In Proposition~\ref{GRRembedding}, we proved that for $l \geq d + e \geq \dim Y+e$ there is an equality 
\[
\ol{f}_*\bigl((l!)^2\cdot T_l^2 \cdot  \ch(x) \cdot Td(T_X)\bigr)=(l!)^2\cdot T_l^2 \cdot \ch(f_*x) \cdot Td(T_{Y\times\PP^e})
\]
in $\CH^*(Y\times\PP^e)$ for any $x\in K_0(X)$. In Proposition~\ref{GRRprojection}, we proved that for $l \geq d + e \geq \dim Y+e$ there is an equality 
\[
p_*\bigl((l!)^2\cdot T_l^2 \cdot  \ch(y) \cdot Td(T_{Y\times\PP^e})\bigr)=(l!)^2\cdot T_l^2 \cdot \ch(p_*y) \cdot Td(T_Y)
\]
in $\CH^*(Y)$ for any $y\in K_0(Y \times \PP^e)$. Combining these two equalities with $y = f_*x$, we obtain the claim. 
\end{proof}

\section{Integral analogue of Kontsevich`s comparison theorem} \label{section:IntKonThm}
In this section we assume that the base field $k$ is perfect, we prove integral Kontsevich`s comparison by using our integral Grothendieck Riemann-Roch theorem. To prove the result, we use the following map called integral Mukai vector.

Let us recall the category $\KM(k)$ of Gillet-Soul\'e’s $K$-motives from \cite[Def.5.1, 5.4, 5.6]{KMOTIVE}. The category $\KM(k)$ is the idempotent completion of the category whose objects are the regular projective $k$-varieties over $k$ and whose morphisms, for regular projective $k$-varieties $X, Y$, are given by
\[
\KM(k)(X,Y) := K_0(X\times Y).
\]
Composition is defined as follows:
\begin{eqnarray*}
\KM(k)(X,Y) \times \KM(k)(Y,Z) & \to& \KM(k)(X\times Z) \\
{[\sF]} \times {[\sG]} &\mapsto& {{p_{13}}_*[p_{12}^*\sF\otimes^{\mathbb{L}}p_{23}^*\sG]}.
\end{eqnarray*}
The identity of $X$ is given by the object in $\KM(k)(X,X)$ corresponding to $[{\Delta_X}_*\sO_X]\in K_0(X\times X)$ where $\Delta_X$ is the diagonal morphism of $X$.

Let $T\in \Chow(k)$ be the Tate motive. Let us recall the orbit category $\Chow(k)/-\otimes T$. The orbit category $\Chow(k)/-\otimes T$ has same objects as $\Chow(k)/-\otimes T$. Its morphisms from $M$ to $N$ are in bijection with 
\[
\bigoplus_{i\in\Z} \hom_{\Chow}(M,N\otimes T^i).
\]
For a morphism $f\in \hom_{\Chow/-\otimes T}(M,N)$, we denote by $f_i$ the $i$-th component of $f$. Given objects $M,N$ and $K$ and morphisms 
\[
f=\{f_i\}_{i\in\Z} \in \bigoplus_{i\in\Z}\hom_{\Chow}(M,N\otimes T^i) \qquad \qquad g=\{g_i\}_{i\in\Z} \in \bigoplus_{i\in\Z}\hom_{\Chow}(N,K\otimes T^i),
\]
the $l$-th component of the composition $g\circ f$ is the finite sum 
\begin{equation}\label{TENSAI}
\sum_{r\in\Z}\left( g_{l-r}\otimes T^r  \right)\circ f_{r}.
\end{equation}
The identity $(id_M)$ of $M$ in $\Chow/-\otimes T$ is given by $id_M\in \hom_{\Chow}(M,M)\subset \hom_{\Chow/-\otimes T}(M,M)$, i.e. ${(id_M)}_0=id_M$ and ${(id_M)}_j=0$ for any $j\neq 0$.
\begin{thm}\label{japanend}
We assume that the base field $k$ is perfect. For natural numbers $d$, $e$, and $l\geq 2d+e$, there is a ${\Z[\frac{1}{(l+1)!}]}$-linear functor 
\[
\Phi_{\Z[\frac{1}{(l+1)!}]}:\KM^{\leq d}_{(e)}(k)_{\Z[\frac{1}{(l+1)!}]} \to \Chow(k)_{\Z[\frac{1}{(l+1)!}]}/-\otimes T_{\Z[\frac{1}{(l+1)!}]}
\]
where $T_{\Z[\frac{1}{(l+1)!}]}$ is the Tate motive. Moreover there is a natural fully faithful functor $\theta:\KM^{\leq d}_{(e)}(k)_{\Z[\frac{1}{(l+1)!}]} \to \KMM(k)_{\Z[\frac{1}{(l+1)!}]}$ and the functor $\Phi_{\Z[\frac{1}{(l+1)!}]}$ makes the following diagram commute:
\[
\xymatrix{
{\textbf{dgcat}}(k) \ar[d]_{U}&&\Sm\Proj^{\leq d}_{(e)}(k)\ar[d]  \ar[ll]_{perf_{dg}(-)} & & \\
\KMM(k)_{\Z[\frac{1}{(l+1)!}]} &&\KM^{\leq d}_{(e)}(k)_{\Z[\frac{1}{(l+1)!}]} \ar[rr]_-{\Phi_{{\Z[\frac{1}{(l+1)!}]}}} \ar[ll]^{\theta} & &  \Chow(k)_{\Z[\frac{1}{(l+1)!}]}/-\otimes T_{\Z[\frac{1}{(l+1)!}]}
}
\]
where $\KMM(k)$ is the category of non-commutative motives (see \cite{Tab14}). After applying $-\otimes \Q$, the lower row is compatible with the fully faithful functor $\Chow(k)_{\Q}/-\otimes T_{\Q} \to \KMM(k)_{\Q}$ described in \cite[section~8]{Tab14}.
\end{thm}
\begin{proof}
For simplicity of notation we write $R$ instead of ${\Z[\frac{1}{(l+1)!}]}$. Let us start with the functor $\theta$. Since $l!$ and $T_l$ are invertible elements in $R$, and $T_l\cdot Td(E)$ and $l!\cdot \ch(E)$ have integral coefficients for any vector bundle $E$ on $X$, thus $\ch(E)$ and $Td(E)$ are well defined in $\CH^*(X)_R$. For a smooth projective variety $X\in\KM^{\leq d}_{(e)}(k)_{R}$, we define $\theta(X)=U(X)$, and for a morphism $a\in \KM(k)(X,Y)_{R}=K_0(X\times Y)_{R}$ we define $\theta(a)=a\in K_0(X\times Y)_{R}=\KMM(U(X),U(Y))$ (see \cite{MarcolliTabuada}). Obviously $\theta$ is a ${R}$-linear fully faithful functor. 

Let us consider $\Phi_{R}$. For a smooth projective variety $X$, we define $\Phi_{R}(X)=\pi(M(X)_R)$. For a morphism $\alpha \in \KM(X\times Y)_R=K_0(X\times Y)_R$, we define 
\[
\Phi_{R}(\alpha)=\tau^{(l)}(\alpha)=\ch(\alpha)\cdot p_{2}^*{Td(T_{ Y})}\in\CH^*(X\times Y)=\Chow(k)_R/-\otimes T_R(X,Y)
\]
where $p_2$ is the projection $X\times Y \to Y$. Let us show that $\Phi_{R}$ satisfies the functoriality conditions. Consider a triple of smooth projective varieties $(X,Y,Z)$ in $\Sm\Proj^{\leq d}_{(e)}(k)$ and correspondences $a\in K_0(X,Y)_{R}$ and $b\in K_0(Y \times Z)_{R}$.
\[
\xymatrix{
&&X \times Y \times Z \ar[ld]_{p_{12}} \ar[rd]^{p_{23}} \ar[rr]^{p_{13}}&&X \times Z\\
&X\times Y \ar[ld]_{p_X} \ar[rd]^{p_Y}&&Y \times Z \ar[ld]_{q_Y}\ar[rd]^{q_Z}&\\
X & & Y & & Z}
\]
We denote the projection $X\times Z \to X$ by $r_X$ and $X\times Z \to Z$ by $r_Z$. By the definition of $\Phi_R$, we have an equality
\begin{eqnarray}\label{1}
\Phi_R(b\circ a) &=& \ch\left( {p_{13}}_*\left( p_{12}^* a \otimes p_{23}^*b \right)\right)\cdot r_Z^*Td(T_Z) \\
&\overset{}{=}& \ch\left( {p_{13}}_*\left( p_{12}^* a \otimes p_{23}^*b \right) \right) \cdot Td(T_{X\times Z}) \cdot  Td(T_{X\times Z})^{-1} \cdot r_Z^*Td(T_Z) \nonumber 
\end{eqnarray}
in $\CH^*(X\times Z)_R$. The morphism p13 can be decomposed in to an immersion $X\times Y\times Z  \to  X \times Z \times \PP^e$ and a projection $X \times Z \times \PP^e \to X \times Z$, thus thanks to integral Grothendieck Riemann-Roch for closed embeddings (Proposition \eqref{GRRembedding}) and projections (Proposition \eqref{GRRprojection}), for $l \geq 2d + e$ we have an equation 
\begin{equation}\label{001}
\ch\left( {p_{13}}_*\left( p_{12}^* a \otimes p_{23}^*b \right) \right) \cdot Td(T_{X\times Z}) =  {p_{13}}_*\Bigl( \ch( p_{12}^* a \otimes p_{23}^*b) \cdot Td(T_{X\times Y \times Z})\Bigr)
\end{equation}
in $\CH^*(X\times Z)_R$. Thus by combining equations \eqref{1} and \eqref{001} we have an equality
\begin{equation}
    \Phi_R(b\circ a)\overset{\eqref{001}}{=}{p_{13}}_*\Bigl( \ch( p_{12}^* a \otimes p_{23}^*b) \cdot Td(T_{X\times Y \times Z})\Bigr)\cdot Td(T_{X\times Z})^{-1} \cdot r_Z^*Td(T_Z)
\end{equation}
in $\CH^*(X\times Z)_R$. Since we know $r_X^*T_X\oplus r_Z^*T_Z= T_{X\times Z}$, we have
\begin{eqnarray*}
\Phi_R(b\circ a) &=&{p_{13}}_*\Bigl( \ch( p_{12}^* a \otimes p_{23}^*b) \cdot Td(T_{X\times Y \times Z}) \Bigr)\cdot r_X^*Td(T_{X})^{-1} \\
    &\overset{\text{proj.form}}{=}&{p_{13}}_*\Bigl( \ch( p_{12}^* a \otimes p_{23}^*b) \cdot Td(T_{X\times Y \times Z}) \cdot p_{13}^*r_X^*Td(T_{X})^{-1}\Bigr).
\end{eqnarray*}
Since we know $p_{12}^*p_{X}^*T_X=p_{13}^*r_X^*T_X$ and 
\[
p_{12}^*p_{X}^*T_X \oplus p_{12}^*p_{Y}^*T_Y \oplus p_{23}^* q_Z^* T_Z= T_{X\times Y \times Z}.
\]
By using this equation, we have
\begin{eqnarray*}
\Phi_R(b\circ a)&=&{p_{13}}_*\Bigl( \ch( p_{12}^* a \otimes p_{23}^*b)\cdot p_{12}^*p_{Y}^*Td(T_Y) \cdot  p_{23}^* q_Z^* Td(T_Z) \Bigr) \\
&=& {p_{13}}_*\Bigl(  p_{12}^*\ch( a ) \cdot p_{23}^*\ch(b)\cdot p_{12}^*p_{Y}^*Td(T_Y) \cdot  p_{23}^* q_Z^* Td(T_Z) \Bigr)  \\
&\overset{\text{reorder}}{=}& {p_{13}}_*\Bigl(  p_{12}^* \bigl(\ch( a )\cdot p_{Y}^*Td(T_Y)\bigr) \cdot p_{23}^*\bigl(\ch(b) \cdot  q_Z^* Td(T_Z)\bigr) \Bigr)\\
&=& {p_{13}}_* (p_{12}^*\Phi_R(a) \cdot p_{23}^* \Phi_R(b)) = \Phi_R(b) \circ \Phi_R(a).
\end{eqnarray*}

Let us prove that $\Phi_R(id_X)=id_{\Phi_R(X)}$ for any $X\in KM^{\leq d}_{(e)}(k)$. For a smooth projective variety $X\in \Sm\Proj^{\leq d}_{(e)}(k)$, we denote by $\Delta$ the diagonal map of $X$. Let us show that $\Phi_{R}({\Delta}_*[\sO_{X}])=[\Delta]$. We denote by $p_1$ (resp. $p_2$) the projection onto the first (resp. second) factor $X\times X \to X$. By the definition of $\Phi_R$, we have an equality
\begin{eqnarray}\label{TAKEZ}
    \Phi_R(\Delta_*[\sO_X])&=&\ch(\Delta_*[\sO_X])\cdot p_2^*Td(T_X)\\
    &=& \ch(\Delta_*[\sO_X]) \cdot Td(T_{X\times X}) \cdot Td(T_{X\times X})^{-1}\cdot p_2^*Td(T_X) \nonumber
\end{eqnarray}
By the Grothendieck Riemann-Roch for closed embedding (Proposition~\ref{GRRembedding}), we have an equation
\[
 \Phi_R(\Delta_*[\sO_X])={\Delta}_*\bigl(\ch([\sO_X])\cdot Td(T_X)\bigr)\cdot Td(T_{X\times X})^{-1}\cdot p_2^*Td(T_X)
\]
in $\CH^*(X\times X)_R$. We know $T_{X \times X} = p_1^*T_X \oplus p_2^*T_X$, thus we have an equation
\[
\Phi_R(\Delta_*[\sO_X])={\Delta}_*\bigl(\ch([\sO_X])\cdot Td(T_X)\bigr)\cdot p_1^*Td(T_X)^{-1}.
\]
By the projection formula \eqref{projTdTd}, we have an equation
\begin{eqnarray*}
\Phi_R(\Delta_*[\sO_X])&=&{\Delta}_*\bigl(\ch([\sO_X])\cdot Td(T_X)\bigr)\cdot p_1^*Td(T_X)^{-1}\\
&\overset{\text{proj.form}}{=}& \Delta_*\Bigl( \ch([\sO_X])\cdot Td(T_X) \cdot \Delta^*p_1^*Td(T_X)^{-1} \Bigr)\\
&\overset{p_1 \circ \Delta=id_X}{=}&\Delta_*\ch([\sO_X])=\Delta_*[X]=[\Delta]
\end{eqnarray*}
in $\CH^*(X\times X)_R$. Thus we obtain the claim.
\end{proof}
\begin{lemma}\label{KSK}
Let $X,Y$ be smooth projective varieties. If there is a fully faithful triangulated $k$-linear functor $F:D^b(X)\hookrightarrow D^b(Y)$, then there is a split injective
\[
X \hookrightarrow Y 
\]
in $\KM(k)$.
\end{lemma}
\begin{proof} 
We consider the following commutative diagram
\[
\xymatrix{
&&X \times Y \times X \ar[ld]_{p_{12}} \ar[rd]^{p_{23}} \ar[rr]^{p_{13}}&&X \times X\\
&X\times Y \ar[ld]_{p_X} \ar[rd]^{p_Y}&&Y \times X \ar[ld]_{q_Y}\ar[rd]^{q_X}&\\
X & & Y & & X}
\]
By \cite{BondalVan}, $F$ has a right adjoint $G$. By \cite[Theorem~3.2.1]{Orlov03}, $F$ can be represented by an object $\sF\in D^b(X\times Y)$ and $G$ can be represented by an object $\sG\in D^b(Y\times X)$. We have an equivalence of functors
\[
Id_{D^b(X)} \simeq G \circ F,
\]
by \cite[Proposition~2.1.2]{Orlov03} and \cite[Theorem~3.2.1]{Orlov03} we have an equality
\begin{equation}\label{SANMA}
    \Delta_*\sO_X={p_{13}}_*\left( p_{12}^*\sF \otimes^{\mathbb{L}} p_{23}^* \sG \right) 
\end{equation}
in $D^b(X\times X)$ where $\Delta$ is the diagonal map of $X$. Consider the morphisms $[\sF] \in \KM(k)(X,Y)=K_0(X\times Y)$ and $[\sG]\in \KM(k)(Y,X) =K_0(Y\times X)$. By the equation~\eqref{SANMA}, we have
\[
id_{X}=[\Delta_*\sO_X]=[\sG]\circ [\sF]
\]
in $\KM(k)(X,X)=K_0(X\times X)$.
\end{proof}
\begin{lemma}\label{UMEHARA}
For smooth projective varieties $X,Y$ in $\Sm\Proj^{\leq d}_{(e)}(k)$and $l\geq 2d+e$, if there is a split injective map 
\[
F:X_{R_{l+1}} \hookrightarrow Y_{R_{l+1}}
\]
in $\KM(k)_{R_{l+1}}$ then the map $\Phi_{R_{l+1}}(F)\in \hom_{\Chow/-\otimes T}(X,Y)$ induces the split injective map
\[
\Phi_{R_{l+1}}(F):M(X)_{R_{l+1}} \hookrightarrow \bigoplus_{i\in\Z}M(Y)_{R_{l+1}}(i)[2i] 
\]
in $\DM(k)_{R_{l+1}}$, where $\DM(k)$ is the triangle category of mixed motives defined by Voevodsky \cite{V00b}.
\end{lemma}
\begin{proof}
We denote by $G$ retraction of $F$. Since we know $G\circ F=id_{X_{R_{l+1}}}$ in $\KM(X,X)$, we have an  equality:
\begin{equation}
    id_{\pi(M(X))}=\Phi_{R_{l+1}}(id_{X_{R_{l+1}}})=\Phi_{R_{l+1}}(G\circ F)=\Phi_{R_{l+1}}(G)\circ \Phi_{R_{l+1}}(F).
\end{equation}
By the definition of compositions of the orbit category \eqref{TENSAI}, we know
\begin{equation}\label{GINGA}
    \sum_{r\in\Z}\left( \Phi_{R_{l+1}}(G)_{l-r}\otimes T^r  \right)\circ \Phi_{R_{l+1}}(F)_{r}= \left\{
\begin{array}{ll}
id_{M(X)_{R_{l+1}}} & l=0 \\
0 & l\neq 0.
\end{array}
\right.
\end{equation}
By this equation, we obtain that the morphism
\[
M(X)_{R_{l+1}} \overset{\{\Phi_{R_{l+1}}(F)_r\}_{r}}{\to} \bigoplus_{i\in\Z}M(Y)_{R_{l+1}}(i)[2i] \overset{\{\Phi_{R_{l+1}}(G)_{-r}\otimes T^r\}_r}{\to } M(X)_{R_{l+1}}
\]
is equal to the identity $id_{M(X)_{R_{l+1}}}$.
\end{proof}
\begin{cor}\label{primeministar}
For smooth projective varieties $X,Y$ in $\Sm\Proj^{\leq d}_{(e)}(k)$and $l\geq 2d+e$, if there is a fully faithful triangulated $k$-linear functor $D^b(X)\hookrightarrow D^b(Y)$, then there is a split injective
\[
M(X)_{R_{l+1}}(d_Y)[2d_Y] \hookrightarrow \bigoplus_{i=0}^{d_X+d_Y}M(Y)_{R_{l+1}}(i)[2i]
\]
in $\DM^\eff(k,R_{l+1})$ where $d_X=\dim X$ and $d_Y=\dim Y$.
\end{cor}
\begin{proof}
By Lemma~\ref{KSK} and Lemma~\ref{UMEHARA}, we have a split injective in $\DM(k,{R_{l+1}})$:
\begin{equation}\label{owari}
M(X)_{R_{l+1}} \hookrightarrow \bigoplus_{i\in Z}M(Y)_{R_{l+1}}(i)[2i].
\end{equation}
For any $i\in\Z$ there is an isomorphism (see \cite{V00b})
\[
\DM(k,R_{l+1})(M(X)_{R_{l+1}},M(Y)_{R_{l+1}}(i)[2i])\simeq \CH^{i+d_Y}(X\times Y)_{R_{l+1}},
\]
thus the map \eqref{owari} induces the desired split injective map of motives
\[
M(X)_{R_{l+1}} \hookrightarrow \bigoplus_{i=-d_Y}^{d_X}M(Y)_{R_{l+1}}(i)[2i].
\]
By the cancellation theorem \cite{cancel}, we have a split injective map
\begin{equation}\label{owari2}
M(X)_{R_{l+1}}(d_Y)[2d_Y] \hookrightarrow \bigoplus_{i=0}^{d_X+d_Y}M(Y)_{R_{l+1}}(i)[2i].
\end{equation}
in $\DM^\eff(k)$.
\end{proof}

\begin{cor}\label{final}
For a smooth projective variety $X$ in $\Sm\Proj^{\leq d}_{(e)}(k)$and $l\geq 2d+e$, if the derived category of $X$ has a full exceptional collection $<E_1,E_2,..,E_m>$, then there is an isomorphism
\[
\bigoplus_{j\in\Z}M(X)_{R_{l+1}}(j)[2j]\simeq \bigoplus_{i=1}^m\bigoplus_{j\in\Z}{R_{l+1}}(j)[2j]
\]
in $\DM(k,R_{l+1})$
\end{cor}
\begin{proof}
The full exceptional collection induces an isomorphism of non-commutative motives (see \cite[Lemma~5.1]{MarcolliTabuada}) 
\[
U(X)_{R_{l+1}}\simeq \bigoplus_{i=1}^mU(k)_{R_{l+1}}.
\]
By Theorem~\ref{japanend} there is an isomorphism in $\KM(k)_{R_{l+1}}$
\[
X_{R_{l+1}}\simeq \coprod_{i=1}^{m} \Spec k_{R_{l+1}}
\]
Applying $R_{l+1}$-linear functor $\Phi_{R_{l+1}}$ to this isomorphism, we obtain an isomorphism in $\DM(k,{R_{l+1}})$:
\[
\bigoplus_{j\in\Z}M(X)_{R_{l+1}}(j)[2j]\simeq \bigoplus_{i=1}^m\bigoplus_{j\in\Z}{R_{l+1}}(j)[2j].
\]
This is what we want.
\end{proof}
By the same discussion as in \cite{MarcolliTabuada}, we provide the explicit formula. We need the following lemmas.
\begin{lemma}\label{trivial}
For a triple $A,B$, and $C$ in $\DM^\eff(k)$, and morphisms $F:A\to B$, $G:B\to C$ and $E:A \to C$ such that $E=G\circ F$, if $E$ is split injective then $F$ is split injective.
\[\xymatrix{
 & A \ar[ld]_{F} \ar[d]^{E}\\
B\ar[r]_{G} & C 
}\]
\end{lemma}
\begin{proof}
We denote by $\tilde{E}$ by the split surjective from $C$ to $A$. Then we have $\tilde{E} \circ G\circ F=id_A$.
\end{proof}
\begin{lemma}\label{trivial2}
For a triple $A,B$, and $C$ in $\DM^\eff(k)$ and split injective maps $F:A \hookrightarrow C$ and $G:B \hookrightarrow C$. Suppose $\hom_{\DM^\eff}(A,B)=0$ and $\hom_{\DM^\eff}(B,A)=0$ then there is a split injective 
\[
A\oplus B \hookrightarrow C.
\]
\end{lemma}
\begin{proof}
We denote by $\tilde{F}$ (resp. $\tilde{G}$) the split surjective $C \to A$ (resp. $C\to B$). Since we assume $\hom_{\DM^\eff}(A,B)=0$ and $\hom_{\DM^\eff}(B,A)=0$, we have $\tilde{G} \circ F=0$ and $\tilde{F} \circ G=0$. Thus the map 
\[
(\tilde{F},\tilde{G}):C \to A\oplus B 
\]
is split surjective.
\end{proof}
\begin{lemma}\label{OOOO}
For a motive $M \in \DM^\eff(k,R)$ and a $\textrm{PID}$ $R$, if there is a split injective map
\[
F:M \hookrightarrow \bigoplus_{i=0}^{d}R(i)[2i]^{\oplus m}
\]
in $\DM^\eff(k,R)$ for some $m\in\N$ and an integer $d$, then there are integer $w\in\N$ and integers $0 \leq i_v\leq d$ for $0\leq v \leq w$ such that there is an isomorphism
\[
M \simeq \bigoplus_{0\leq v \leq w}R(i_v)[2i_v]
\]
in $\DM^\eff(k,R)$.
\end{lemma}

\begin{proof}
We note that there is an equality of $R$-modules
\begin{equation}\label{FUFU}
\DM^\eff(k,R)(R(i)[2i],R(j)[2j])=\left\{
\begin{array}{ll}
R & i=j \\
0 & i\neq j.

\end{array}
\right.
\end{equation}

Firstly we prove that for an object $C\in\DM^\eff(k,R)$ satisfying that there is a split injective 
\[
C\hookrightarrow \bigoplus_{i=0}^{l} R(i)[2i]^{\oplus m}
\]
and 
\begin{equation}\label{condition2}
\DM^\eff(R(i)[2i],C)=0,
\end{equation}
for any $0\leq i \leq l$ then $C=0$. We denote by $G$ the split surjective map $\bigoplus_{i=0}^{l} R(i)[2i]^{\oplus m} \to C$. For any projection on the coordinate $\pi:R(j)[2j] \hookrightarrow \bigoplus_{i=0}^{d}R(i)[2i]^{\oplus m}$, by the condition~\eqref{condition2}, we know $G\circ \pi=0$. Thus the split surjective map $G=0$, i.e. $C=0$.

Let us prove the claim. We fix an integer $0\leq i \leq l$. The split injective map $F$ induces a split injective morphism of $R$-modules
\begin{equation}\label{HOSHIN}
\hom_{\DM^\eff}(R(i)[2i],M)\hookrightarrow \hom_{\DM^\eff}(R(i)[2i], \bigoplus_{j=0}^{l} R(j)[2j]^{\oplus m}) \overset{\eqref{FUFU}}{\simeq} R^{\oplus m},
\end{equation}
and since $R$ is PID, thus there is an integer $0 \leq m_i \leq m$ such that there is an isomorphism of $R$-modules
\begin{equation}\label{TOSHI}
   \hom_{\DM^\eff}(R(i)[2i],M)\simeq R^{\oplus m_i}.
\end{equation}
By \eqref{HOSHIN}, \eqref{TOSHI}, the split injective map $F$ induces a split injective map
\begin{equation}\label{PURIO}
    \tilde{F_i}:R^{\oplus m_i}\overset{\eqref{TOSHI}}{\simeq} \hom_{\DM^\eff}(R(i)[2i],M) \overset{\eqref{HOSHIN}}{\hookrightarrow} R^{\oplus m}.
\end{equation}
We denote by $\mathbf{e}_j\in 
R^{\oplus m_i}$ the $j$-th coordinate, i.e. $\mathbf{e}_j =(0,0,...,\overset{j}{1},..0)$. Since $\tilde{F}$ is split injective, there are elements $r_{j}\in R^{\oplus m}$ for $m_i+1 \leq j \leq m$ such that a set 
\[
\{\tilde{F}(\mathbf{e}_1),\tilde{F}(\mathbf{e}_2),..,\tilde{F}(\mathbf{e}_{m_i}),r_{m_i+1},..,r_{m}\}
\]
is a basis of $R^{\oplus m}$. 

We take 
\[
\oplus_{j=1}^{m_i} \mathbf{e}_j\in   \bigoplus_{j=1}^{m_i}R^{\oplus m_i}\overset{\eqref{TOSHI}}{\simeq}\bigoplus_{j=1}^{m_i} \hom_{\DM^\eff}(R(i)[2i],M) = \hom_{\DM^\eff}(R(i)[2i]^{\oplus m_i},M)
\]
and, by using isomorphism \eqref{FUFU}, we take
\[
E:R(i)[2i]^{\oplus m_i} \to R(i)[2i]^{\oplus m} \qquad (0,..,0,\overset{j}{1},0,..,0) \mapsto \tilde{F}(\mathbf{e}_j).
\]
Then there is a following commutative diagram:
\[
\xymatrix{
   &  R(i)[2i]^{\oplus m_i} \ar[d]^{E}\ar[ldd]_{\oplus_{j=1}^{m_i} \mathbf{e}_j}\\
   &  R(i)[2i]^{\oplus m} \ar[d]^{\text{nat.map}} \\
 M \ar[r]_-{F} & \bigoplus_{j=0}^lR(j)[2j]^{\oplus m} 
}\]
because for the projection on the $j$-th coordinate $\pi_j:R(i)[2i] \to R(i)[2i]^{\oplus m_i}$ we have
\[
F\circ\oplus_{j=1}^{m_i} \mathbf{e}_j\circ \pi_j=F\circ \mathbf{e}_j\overset{\eqref{PURIO}}{=}\text{nat.map}\circ \tilde{F}(\mathbf{E}_j)
\]
in $\hom_{\DM^\eff}(R(i)[2i],\bigoplus_{j=0}^lR(j)[2j]^{\oplus m})$. By lemma~\ref{trivial}, the map $\oplus_{j=1}^{m_i} \mathbf{e}_j$ is split injective since $E$ is split injective. We denote $\oplus_{j=1}^{m_i} \mathbf{e}_j$ by $P_i$.

Let us consider a map
\[
P=\oplus_{i=0}^l P_i:  \oplus_{i=0}^lR(i)[2i]^{\oplus m_i} \to M.
\]
Since each of $P_i$ is split injective, by Lemma~\ref{trivial2} and \eqref{FUFU}, the map $P$ is split injective. 

Consider $C:=\Cone(P:\oplus_{i=0}^lR(i)[2i]^{\oplus m_i} \to M)$. Since $M$ is a direct summand of $\bigoplus_{i=0}^{l} R(i)[2i]^{\oplus m}$, the object $C$ is also a direct summand of $\bigoplus_{i=0}^{l} R(i)[2i]^{\oplus m}$. For any $0\leq i \leq l$, by the defition of $P_i$ we have an isomorphism
\[
\hom(R(i)[2i],P):\hom(R(i)[2i],\oplus_{i=0}^lR(i)[2i]^{\oplus m_i}) \simeq \hom(R(i)[2i],M).
\]
Thus $\hom_{\DM^\eff}(R(i)[2i],C)=0$. Thus $C=0$.
\end{proof}

\begin{cor}
For a smooth projective variety $X$ in $\Sm\Proj^{\leq d}_{(e)}(k)$and $l\geq 2d+e$, if the derived category of $X$ has a full exceptional collection $<E_1,E_2,..,E_m>$, then there  is  a  choice  of  integers $r_1,r_2,...,r_m\in \{0,...,\dim X\}$ giving rise to a canonical isomorphism
\[
M(X)_{\Z[\frac{1}{(l+1)!}]}\simeq \bigoplus_{i=1}^m{\Z[\frac{1}{(l+1)!}]}(r_i)[2r_i].
\]
\end{cor}
\begin{proof}
By Corollary~\ref{final}, we have a split injective map
\[
M(X)_{R_{l+1}} \hookrightarrow \bigoplus_{j=1}^{m} \bigoplus_{i\in\Z}R_{l+1}(i)[2i].
\]
By the discussion in the proof of Corollary~\ref{primeministar}, the map induces a split injective map
\[
M(X)_{R_{l+1}} \hookrightarrow \bigoplus_{j=1}^{m} \bigoplus_{i=0}^{d_X}R_{l+1}(i)[2i].
\]
By Lemma~\ref{OOOO}, there is an isomorphism
\[
M(X)_{R_{l+1}}\simeq \bigoplus_{v=1}^{w}R_{l+1}(i_v)[2i_v]
\]
    for some integers $0 \leq i_v \leq d_X$ and $w\in \N$. We will show that $w=m$. By Theorem~\ref{japanend} we have
\[
\pi(M(X)_{R_{l+1}}) \simeq \bigoplus_{i=1}^{m}\pi(R_{l+1})
\]
and $\pi(R_{l+1}(i_v)[2i_v])\simeq \pi(R_{l+1})$, we have
\[
\pi(M(X)_{R_{l+1}})\simeq  \bigoplus_{v=1}^{w} \pi(R_{l+1}(i_v)[2i_v]) \simeq  \bigoplus_{v=1}^{m} \pi(R_{l+1}).
\]
Thus we obtain $w=m$.\end{proof}

\appendix
\section{Pappas's integral Grothendieck Riemann-Roch theorem}
In this section, we recall Pappas's integral Grothendieck Riemann-Roch theorem \cite{Papas07}, and we show that in the case $\ch(k)=0$ Theorem~\ref{GRRfinite} is deduced by Pappas's integral Grothendieck Riemann-Roch. Let $X$ be a smooth variety, $E$ be a vector bundle. Let $a_1,.,a_r$ be Chern roots of $E$. For any $m$, we will denote the degree $m$ part of the symmetric polynomial $Td(E) \in \Q[a_1,...,a_r]^{Sym(r)}$ by $Td_m(E)$. By \cite[Lemma~1.7.3]{Hi}, the polynomial
\[
\mathfrak{Td}_m:=T_mTd_m(E)
\]
has integral coefficients, where we set $T_m:=\displaystyle\prod_{p:\text{prime number}} p^{[\frac{m}{p-1}]}$. By definition, we then have 
\[
T_{l}\cdot Td(E)=\sum_{m=0}^{\dim X}\frac{T_l}{T_m}\mathfrak{Td}_m
\]
in $\CH^*(X)$ for any $l \geq \dim X $.
For any $m\in\N$, we will also denote the degree $m$ part of the symmetric polynomial $\ch(E) \in \Q[a_1, \dots, a_r]^{Sym(r)}$ by $\ch_m(E)$. We set 
\[
\mathfrak{s}_m(E):=m!\ch_m(E).
\]
giving identities
\[
l!\cdot \ch(E)=\sum_{m=0}^{\dim X}\frac{l!}{m!}\mathfrak{s}_m
\]
in $\CH^*(X)$ for any $l \geq \dim X $. Set the polynomial
\[
\mathfrak{CT}_m(E):=T_m\cdot (\ch(E)Td(T_X))_m=\sum_{j=0}^{m}\frac{T_m}{j!\cdot T_{m-j}}\cdot(\mathfrak{s}_j(E) \mathfrak{Td}_{m-j}(T_X))
\]
in $\Q[a_1,...,a_r]^{Sym(r)}$, where $(\ch(E)Td(T_X))_m$ is the degree $m$ part of the polynomial $\ch(E)Td(T_X)$. By \cite[Lemma~2.1]{Papas07}, $\mathfrak{CT}_m$ is a homogeneous polynomial in $\Z[a_1,...,a_r]$. For any $l \geq \dim X$ there exist an identity
\begin{equation}\label{beastseven}
T_l\cdot \ch(E)Td(T_X)= \sum_{m=0}^{\dim X}\frac{T_l}{T_m}\mathfrak{CT}_m(E)
\end{equation}
in $\CH^*(X)$.

\begin{thm}[Integral Gorthendieck Riemann-Roch {\cite[Theorem~2.2]{Papas07}}]\label{PapasintegralGRR}
Suppose $k$ is a field of characteristic $0$. Let $X$ and $S$ be smooth quasi-projective varieties over $k$ and let $f:X \to S$ be a projective morphism over $k$. Set $d=d_f=\dim X-\dim S$ and suppose that $\sF$ is a coherent sheaf on $X$.\begin{itemize}
    \item[(a)] If $d\geq 0$ then 
    \[
    \frac{T_{d+n}}{T_n}\mathfrak{CT}_n(f_*[\sF]) = f_*(\mathfrak{CT}_{d+n}(\sF))
    \]
    holds in $\CH^n(S)$.

     \item[(b)] If $d < 0$ then 
     \[
    \mathfrak{CT}_n(f_*[\sF]) =\frac{T_{n}}{T_{d+n}}\ f_*(\mathfrak{CT}_{d+n}(\sF))
    \]
    holds in $\CH^n(S)$.
\end{itemize}
\end{thm}

Over a characteristic zero base field, Pappas's integral Grothendieck Riemann-Roch leads to the following strengthening of our integral Grothendieck Riemann-Roch Theorem~\ref{GRRfinite} (our theorem needs $l \geq d + e$ where $X, Y$ are of dimension $\leq d$ and embeddable in $\PP^e$).

\begin{thm}\label{PAPAKONT}
We assume that the base field $k$ is of characteristic $0$. Let $f:X \to Y$ be a projective morphism of smooth projective varieties over a field $k$. For any $x \in K_0(X)$, there exists an equality 
\[
f_*\bigl(T_l \cdot \ch(x)Td(T_X)\bigr) = T_l \cdot \ch(f_*x)Td(T_Y)
\]
in $\CH^*(Y)$ for any $l\geq \max\{\dim X,\dim Y\}$.
\end{thm}

\begin{proof}
We denote $\dim X$ by $d_X$ and $\dim Y$ by $d_Y$. For any $l \geq \max\{\dim X, \dim Y\}$, in the case $d\geq 0$, if we multiply equations in Theorem~\ref{PapasintegralGRR} by $\frac{T_l}{T_{d+n}}\in \Z$, then we have
\[
    \frac{T_{l}}{T_n}\mathfrak{CT}_n(f_*x) = \frac{T_{l}}{T_{d+n}} f_*(\mathfrak{CT}_{d+n}(x))
\]    
in $\CH^*(Y)$. Thus by the equation \eqref{beastseven}, we obtain the claim. In the case $d < 0$, if we multiply equations in Theorem~\ref{PapasintegralGRR} by $\frac{T_l}{T_{n}}\in \Z$ we have
\[
\frac{T_{l}}{T_{n}}\mathfrak{CT}_n(f_*x) =\frac{T_{l}}{T_{d+n}}\ f_*(\mathfrak{CT}_{d+n}(x))
\]
Thus by the equation \eqref{beastseven}, we obtain the claim.
\end{proof}
We denote by $\Sm\Proj^{\leq d}(k)$ the full subcategory of $\Sm\Proj(k)$ such that dimension of objects is less than or equal to $d$. By the same discussion in the proof of Theorem~\ref{japanend}, Theorem~\ref{PAPAKONT} induces the following.  
\begin{cor}\label{A33}
We assume that the base field $k$ is of characteristic $0$. For natural numbers $d$ and $l\geq 3d$, there is a ${\Z[\frac{1}{(l+1)!}]}$-linear functor 
\[
\Phi_{\Z[\frac{1}{(l+1)!}]}:\KM^{\leq d}_{(e)}(k)_{\Z[\frac{1}{(l+1)!}]} \to \Chow(k)_{\Z[\frac{1}{(l+1)!}]}/-\otimes T_{\Z[\frac{1}{(l+1)!}]}
\]
where $T_{\Z[\frac{1}{(l+1)!}]}$ is the Tate motive. The functor $\Phi_{\Z[\frac{1}{(l+1)!}]}$ making the following diagram commute:
\[
\xymatrix{
{\textbf{dgcat}}(k) \ar[d]_{U}&&\Sm\Proj^{\leq d}_{(e)}(k)\ar[d] \ar[rr] \ar[ll]_{perf_{dg}(-)} & & \Sm\Proj(k) \ar[d]^{\pi(M(-))}\\
\KMM(k)_{\Z[\frac{1}{(l+1)!}]} &&\KM^{\leq d}_{(e)}(k)_{\Z[\frac{1}{(l+1)!}]} \ar[rr]_-{\Phi_{{\Z[\frac{1}{(l+1)!}]}}} \ar[ll]^{\theta} & &  \Chow(k)_{\Z[\frac{1}{(l+1)!}]}/-\otimes T_{\Z[\frac{1}{(l+1)!}]}
}
\]
where $\theta$ is the natural fully faithful functor. After applying $-\otimes \Q$, the lower row is compatible with the fully faithful functor $\Chow(k)_{\Q}/-\otimes T_{\Q} \to \KMM(k)_{\Q}$ described in \cite[section~8]{Tab14}.
\end{cor}

\bibliography{bib}
\bibliographystyle{alpha}

%

%

\end{document}